\documentclass[11pt]{article}

\usepackage[a4paper,left=2cm,right=2cm,top=1.5cm,bottom=2cm]{geometry}
\usepackage[utf8]{inputenc}
\usepackage[T1]{fontenc}
\usepackage[english]{babel}
\usepackage{amsthm}
\usepackage{amsmath}
\usepackage{amssymb}
\usepackage{mathrsfs}
\usepackage{setspace}
\usepackage{float}
\usepackage{multirow}
\usepackage{stmaryrd}
\usepackage{color}
\usepackage{wrapfig}
\usepackage{pifont}
\usepackage{array}
\usepackage[colorlinks=true,linkcolor=ocre,citecolor=ocre]{hyperref}
\usepackage{dsfont}
\usepackage{upgreek}
\usepackage{nicefrac}
\usepackage{tikz}{}
\usepackage{gensymb}
\usepackage{FiraSans}
\usepackage{graphicx}
\usepackage{caption}
\usepackage{dsfont}
\usepackage{mathtools}
\usepackage{cancel} 

\DeclareCaptionFormat{custom}
{%
    #1#2\textit{\small #3}
}
\usepackage{mathtools}
\numberwithin{equation}{section}
\mathtoolsset{showonlyrefs}

\renewenvironment{proof}{\noindent{\sffamily{\textbf{Proof :}}}}{\begin{flushright}$\square$\end{flushright}}

\newcommand{\IN}{\mathbb{N}}
\newcommand{\IZ}{\mathbb{Z}}

\newcommand{\IR}{\mathbb{R}}
\newcommand{\IC}{\mathbb{C}}
\newcommand{\IT}{\mathbb{T}}

\newcommand{\drm}{\mathrm d}

\newcommand{\CS}{\mathcal S}
\newcommand{\CT}{\mathcal T}

\newcommand{\CE}{\mathcal E}

\newcommand{\Id}{\mathrm{Id}}
\renewcommand{\Re}{\mathrm{Re}}

\newcommand{\enstq}[2]{\left\{#1~\middle|~#2\right\}}

\newcommand{\eps}{\varepsilon}

\usepackage{relsize}

\usepackage{tikz-cd}

\tikzset{baseline,tdot/.style={circle,draw,inner sep=1pt,fill},ndot/.style={circle,draw,inner sep=1pt},ndotd/.style={circle,draw,inner sep=0.5pt},cdot/.style={circle split,draw,inner sep=0.7pt},ddot/.style={circle split,draw,inner sep=0.7pt,rotate=90},tree/.style={inner sep=1pt},odot/.style={shape=coordinate,circle,draw,inner sep=0.05pt,fill},rdot/.style={shape=coordinate,rectangle,toon,draw,inner sep=1pt,outer sep=1.5pt,fill},gdot/.style={shape=coordinate,circle,draw,inner sep=1pt,outer sep=1.5pt,fill},ngdot/.style={shape=coordinate,circle,draw,inner sep=1pt,outer sep=1.5pt},pi/.style={decorate,decoration={zigzag,amplitude=2pt,pre length=4pt, post length=4pt,segment length=3pt}}}

\newcommand{\cut}{t}

\definecolor{ocre}{RGB}{64,123,121}

\newcounter{item}
\numberwithin{item}{section}

\newtheorem{theorem}[item]{\sffamily Theorem}
\newtheorem{definition}[item]{\sffamily Definition}
\newtheorem{proposition}[item]{\sffamily Proposition}
\newtheorem{lemma}[item]{\sffamily Lemma}
\newtheorem{corollary}[item]{\sffamily Corollary}

\newtheorem*{theorem*}{\sffamily Theorem}
\newtheorem*{definition*}{\sffamily Definition}
\newtheorem*{proposition*}{\sffamily Proposition}
\newtheorem*{lemma*}{\sffamily Lemma}
\newtheorem*{corollary*}{\sffamily Corollary}

\usepackage[explicit]{titlesec}
\titleformat{\section}{\centering\Large\bfseries}{\thesection \ --}{0.7em}{\Large\bfseries #1}
\titleformat{\subsection}{\centering\large\bfseries}{\thesubsection \ --}{0.4em}{\large\bfseries #1}
\titleformat{\subsubsection}{\centering\bfseries}{\thesubsubsection \ --}{0.4em}{\bfseries #1}

\let\emph\relax
\DeclareTextFontCommand{\emph}{\bfseries\em}

\title{\bfseries High order uniform in time schemes for weakly nonlinear Schrödinger equation and wave turbulence}
\author{\textsc{Quentin CHAULEUR} and \textsc{Antoine MOUZARD}}
\date{}

\begin{document}

\maketitle
\abstract{We introduce two multiscale numerical schemes for the time integration of weakly nonlinear Schrödinger equations, built upon the discretization of Picard iterates of the solution. These high-order schemes are designed to achieve high precision with respect to the small nonlinearity parameter under particular CFL condition. By exploiting the scattering properties of these schemes thanks to a low-frequency projected linear flow, we also establish its uniform accuracy over long time horizons. Numerical simulations are provided to illustrate the theoretical results, and these schemes are further applied to investigate dynamics in the framework of wave turbulence.}
\vspace{0.5cm}

\section{Introduction}

We consider the nonlinear Schrödinger equation
\begin{equation}\tag{NLS}\label{NLS}
i\partial_t u+\Delta u=\eps|u|^{p-1} u
\end{equation}
with initial data $u(0)=\varphi$ on the full space $\IR^d$ for $1 \leq d \leq 3$. We restrict our attention to defocusing nonlinearity with odd exponents $p\in 2\IN +1$, with furthermore assumption $p\ge5$ in dimension $d=1$ and only the cubic case $p=3$ for $d=3$. In our framework, the nonlinearity strength $\eps>0$ is considered small, putting ourselves in the so-called \textit{weakly nonlinear} regime. Our goal is to investigate and design semi-discrete in time schemes which can capture the multiscale behavior of equation \eqref{NLS} with respect to this small nonlinear strength. Moreover, such schemes will prove to have uniform in time error, that is independent of $T>0$ the horizon time.

\smallskip

The nonlinear Schrödinger equation is a well-studied fundamental model, with a lot of physical applications including Bose-Einstein condensation or nonlinear optics. Our motivation here lies in the simulation of wave turbulence phenomenon, which can be observed when a large number of nonlinearly interacting waves of varying wavelengths propagate in multiple directions, such as in oceanography for which equation \eqref{NLS} stands as a toy model. For the past years, this topic has been the subject of intense mathematical activites with different directions, see for instance the work of \textsc{Deng} and \textsc{Hani} \cite{DengHani2023} and references therein. The usual setting is to consider equation~\eqref{NLS} on a large torus $\IT_L^d$ in the limit $L\gg1$ and $\eps\ll1$ with appropriate scaling laws. On the other hand, our framework is motivated by the recent work of \textsc{Faou} and \textsc{Mouzard}~\cite{FaouMouzard2024} which consider the full space $\IR^2$ with small initial data in the weighted space
\begin{equation}
\Sigma \coloneqq \enstq{\varphi(x)\in H^1(\IR^2)}{\ |x|\varphi(x)\in L^2(\IR^2)}
\end{equation}
with a structure that mimics the large torus $\IT_L^2$. This space is natural in the context of the nonlinear Schrödinger equation to obtain scattering results, which roughly states that the solution behaves in large times as a solution of the free equation. 

\smallskip 

In the different approaches for wave turbulence problems, the main idea is to iterate the Duhamel formulation of the equation to get an expansion of the solution with respect to $\eps>0$, which gives the formal series
\begin{equation}\label{expansioneps}
u(t,x)=\sum_{n\ge0}\eps^nU_n(t,x)
\end{equation}
where $U_n$ is a $(p+(p-1)n)$-linear functional of the initial data. For instance, \textsc{Deng} and \textsc{Hani} considered in \cite{DengHani2023} such arbitrary high-order expansion for the cubic equation $p=3$ in dimension $d\ge3$. They identify in the large number of terms the main contributions that lead to a kinetic description of the covariances of the Fourier modes for random initial data. In a different direction, \textsc{Faou} and \textsc{Mouzard} consider in \cite{FaouMouzard2024} first and second order expansions to identify a kinetic operator for both deterministic and random initial data. In this work, we propose a numerical scheme that captures this multiscale structure with respect to $\eps>0$ for initial data in weighted Sobolev spaces~$\Sigma\cap H^\sigma$ with $\sigma$ large enough. In particular, the terms $U_n$ are defined by a recursive formula and have a tree-like structure that we exploit to construct arbitrary high order schemes. Plugging the series \eqref{expansioneps} into equation \eqref{NLS}, the first order iterate corresponds to the linear equation
\[ i\partial_t U_0 + \Delta U_0 = 0  \]
with initial condition $U_0(0)=\varphi$, which solution is explicitly given by the linear flow $U_0(t)=e^{it\Delta}\varphi$. Following iterates then satisfy the cascades of equations
\begin{equation} \label{eq_Un}
i\partial_t U_{n}+\Delta U_{n}=F_{n}
\end{equation}
with $U_{n}(0)=0$ and
\begin{equation}\label{def_Fn}
F_{n} \coloneqq \sum_{n_1+\ldots+n_p=n-1} U_{n_1} \overline{U_{n_2}}U_{n_3}  \ldots \overline{U_{n_{p-1}}} U_{n_p}
\end{equation}
for all $n\geq1$. Since the sum only involves terms of strictly lower order, $U_{n}$ is given by a linear Schrödinger equation with a forcing that depends on all lower order terms. The main contribution of this work is to design uniform in time numerical schemes for the truncated family $(U_n)_{0\le n\le N-1}$ of order $\eps^N$ for any $N\ge1$.

\smallskip

This work strongly relies on the scattering result from \textsc{Carles} and \textsc{Gallagher} \cite{Carles2009} which controls the error between the solution to \eqref{NLS} and the expansion \eqref{expansioneps} in the space $\Sigma$. Namely, it ensures the bound
\begin{equation}\label{ScatteringBound}
\big\|u(t)-\sum_{n=0}^{N-1}\eps^nU_n(t)\big\|_\Sigma\lesssim \eps^N\|\varphi\|_{\Sigma}^{p+N(p-1)}
\end{equation}
hence a precise time discretization for each $U_n$ for $n<N$ is enough to obtain a numerical scheme with error $\eps^N$ for the solution $u$. For example, the linear flow is an approximation of the solution up to an error of order $\eps$ while the next terms improve more and more the description of the nonlinear behavior. Moreover, due to the expansion in the small parameter $\eps$, each term $U_n$ can be computed with decreasing accuracy as $n$ increases, specifically with an error of order $\eps^{N-n}$. Given that each $U_n$  satisfies a linear Schrödinger equation with a forcing term dependent on the preceding iterates, it is essential to ensure a consistent discretization across different values of $n$. This leads us to introduce the concept of \textit{nested} schemes, further developed in Section \ref{sec:nested}, which ensures such coherence. The mild formulation of \eqref{eq_Un} gives
\begin{equation}
U_n(t)=-i\int_0^te^{i(t-s)\Delta}F_n(s)\drm s
\end{equation}
for $n\ge1$ since $U_n(0)=0$ which we approximate with a time discretization. For fixed $T>0$, we consider a discretization mesh $\tau>0$ such that $J\tau=T$ for $J\in\IN$ and the grid $t_j=j\tau$ for $0\le j\le J$. We then write
\begin{equation}
U_n(t_j)=-i\sum_{\alpha=0}^{j-1}\int_{t_\alpha}^{t_{\alpha+1}}e^{i(t_j-s)\Delta}F_n(s)\drm s
\end{equation}
and use two different methods to approximate the time integral on each subinterval $[t_\alpha,t_{\alpha+1}]$. Since each $F_n(s)$ depends on all lower order terms $U_{n'}$ with $n'<n$, a careful propagation of error is required. At this stage, we also observe that the small parameter  $\eps>0$ no longer appears explicitly, which in turn imposes a CFL-type condition on the time step $\tau$. Our objective is to ensure an error of order $\tau^{N-n}$ for each $U_n$, which motivates the nested structure of our numerical schemes.

\smallskip

Another significant challenge in our work is to construct uniform in time numerical scheme using scattering, as done with the recent approach of \textsc{Carles} and \textsc{Su} \cite{Carles2024}. A key aspect is the formulation of the dispersive equation \eqref{NLS} in the scattering space $\varphi \in \Sigma$. This perspective is particularly relevant given the physical motivation of our study, which is rooted in wave turbulence theory. Indeed, in the absence of dissipation or external forcing, solutions of \eqref{NLS} are expected to exhibit interesting transient dynamics over the so-called \textit{kinetic time}, which grows as $\eps$ tends to~$0$. Additionally, \cite{Carles2024} establishes that a first order filtered Lie splitting method satisfies such uniform estimates, whereas extending this property to higher order schemes is far from straightforward. In this context, our work can be viewed as a natural continuation of  \cite{Carles2024} within a weakly nonlinear framework.

\smallskip

We stress out that computing long time behaviors of nonlinear Schrödinger equations has been an intensive 	field of research in the past decades, and we refer to the book of \textsc{Faou} \cite{Faou2012} and references within. A lot of methods based on Birkhoff normal form techniques and modulated Fourier expansions \cite{FaouGrebert2010,FaouGrebert2010bis,Gauckler2010,GaucklerLubich2010} have proven to be very efficient on time scales of order~$T=\mathcal{O}(\eps^{-N})$ on periodic domains. Note that regularity compensation oscillation technique have also recently been introduced in \cite{BaoCaiFeng2023} for the same purposes. We also point that recent numerical studies have been performed in \cite{DuBuhler2023} and \cite{YangGroomsJulien2021} for wave turbulence problems.

\smallskip

While we consider the case of NLS equation with small non-linearity, we believe that the multiscale numerical schemes introduced here could be useful for many other problems. A first example also motivated by turbulence is the linear Schrödinger equation
\begin{equation}
i\partial_tu+\Delta u=\eps Vu
\end{equation}
with a potential $V:\IR^d\to\IR$, possibly random or time dependent, see for example the works of \textsc{Erdos} and \textsc{Yau} \cite{ErdosYau2000} in the first case or \textsc{Maspero} and \textsc{Robert} \cite{MasperoRobert2017} in the second one. Finally, the decorated tree structure of the iterates of the NLS equation also naturally appears for numerical scheme in the different direction of low regularity initial data performed by \textsc{Bruned} and \textsc{Schratz}~\cite{BrunedSchratz2022} and following works.

\medskip

This paper is organized as follows. In Section \ref{sec:nested}, we introduce our two numerical schemes, relying respectively on high order Newton-Cotes quadrature methods for integration with a coherent families of grids to ensure different precisions levels for each $U_n$, and high-order Taylor expansions of the solutions around temporal grid points. We also state our convergence results for both schemes, namely Theorem \ref{theorem_convergence_NQS} and Theorem \ref{theorem_convergence_NTS}. In Section \ref{section_dispersive_estimates}, we give the continuous and discrete dispersive estimates needed to prove the convergence results in respectively Section \ref{sec:convergence_NQS} for the nested quadrature scheme~\eqref{NQS} and Section \ref{sec:convergence_NTS} for the nested Taylor scheme~\eqref{NTS}. Finally, in Section~\ref{sec:numerical_experiments}, we implement and illustrate the convergence of both schemes and apply them in the context of wave turbulence in Section \ref{sec:wave_turbulence}. Note that all codes are available on the Gitlab page \url{https://plmlab.math.cnrs.fr/chauleur/codes/}.

\smallskip

For clarity and brevity, we denote space norms associated to Lebesgue spaces $L^p(\IR^d)$ and Sobolev spaces $W^{\sigma,p}(\IR^d)$ by $\|\cdot\|_{L^p_x}$ and $\|\cdot\|_{W^{\sigma,p}_x}$ in mathematical mode, respectively. Similarly, time norms are written as $\|\cdot\|_{L^p_t}$ for continous time and $\|\cdot\|_{\ell^p_{\tau}}$ for discrete time. Note that we may indicate the dependence on either continuous time $t$ or discrete times $t_j$ within these norms, in order to clarify the context and make the notation more transparent. The symbol $\lesssim$ denotes an inequality up to a constant that may depend on various parameters of the analysis, but remains uniform with respect to the time horizon $T$ and nonlinear strength $\eps$.

\medskip

\textbf{Acknowledgments:} Q.C. acknowledges the support of the CDP C2EMPI, together with the French State under the France-2030 programme, the University of Lille, the Initiative of Excellence of the University of Lille, the European Metropolis of Lille for their funding and support of the R-CDP-24-004-C2EMPI project. The authors wish to thank Geoffrey Beck, Laurent Chevillard and Giorgio Krstulovic for enlightening discussions on wave turbulence theory and especially about power-law solutions of the Wave Kinetic Equation. The authors are also grateful to Rémi Carles for highlighting the use of the operator $J(t)$, which allowed us to broaden the scope of our results.

\section{Nested formulas and main results} \label{sec:nested}

We adopt the Fourier transformation convention
\begin{equation}  
\widehat{\phi}(\xi)=\int_{\IR^d} \phi(x) e^{-i x \cdot \xi} \drm x
\end{equation}
for all $\xi \in \IR^d$ and we denote by 
\begin{equation}
S(t)\phi \coloneqq e^{it\Delta}\phi
\end{equation}
the linear flow of the Schrödinger equation for any $\phi \in L^2(\IR^d)$ and $t \in \IR$. We also denote by
\begin{equation} 
S_{\tau}(t) \phi \coloneqq S(t) \Pi_{\tau} \phi
\end{equation}
the filtered linear flow with low-frequency projector
\begin{equation}
\widehat{\Pi_{\tau} \phi} (\xi) = \widehat{\phi}(\xi) \chi(\sqrt{\tau}\xi)
\end{equation}
for any $\tau>0$ and for a given cut-off function $\chi \in \mathcal{C}^{\infty}(\IR^d)$ supported on $B^d(0,2)$ such that $\chi \equiv 1$ on~$B^d(0,1)$. Note that low frequency projected scheme has proven to be very efficient in the context of low regularity schemes \cite{IgnatZuazua2009,Ignat2011,ChoiKoh2021,OstermannRoussetSchratz2021}, and more recently for uniform in time scheme \cite{Carles2024}. For any~$n\ge1$, Duhamel's formulation of equation \eqref{eq_Un} gives
\begin{equation}\label{eq_Un_integrated}
U_n(t)=-i\int_0^t S(t-s) F_n(s)\drm s
\end{equation}
for any $t\in\IR$. To lighten the notation, we might omit complex conjugation in the following of the terms where it has no impact on the computations. 

\medskip

We now fix an horizon time $T>0$ and an order of convergence $N\in\IN^*$. We denote by $\tau>0$ the time step and $J\in\IN$ such that $T=J\tau$ and consider the discretization
\begin{equation}
t_j=j\tau
\end{equation}
of the time interval $[0,T]$ for $0\le j\le J$. In the next sections, we will construct two families of numerical scheme that approximate $U_n(t_j)$ for $0\le n\le N-1$ and $0\le j\le J$ with an error of order~$\tau^{N-n}$. These two schemes will be respectively based on a quadrature discretization of the integral appearing in the Duhamel equations \eqref{eq_Un_integrated}, and on a high-order Taylor expansion of $F_n$ around the discretization points $t_j$.

\subsection{Nested Quadrature Scheme}

We first introduce the Nested Quadrature Scheme \eqref{NQS} with nested discretization grids. Given a smooth function $f:[0,T]\to\IR$, the idea of quadrature formulas is to approximate $f$ by a polynomial~$P_j^m$ of degree $m\in\IN$ on each time interval $[t_j,t_{j+1}]$. To do that, introduce the finer grid
\begin{equation}
t_{j,\beta}^{(m)} \coloneqq t_j+\frac{\beta}{m}\tau
\end{equation}
for $0\le \beta\le m$. The Newton-Cotes quadrature method of order $m$ is the approximation
\begin{equation}
\int_0^T f(t)\drm t\simeq\sum_{j=0}^{J-1}\int_{t_j}^{t_{j+1}}P_j^m(t)\drm t
\end{equation}
with $P_j^m$ the unique polynomial of degree $m$ such that $P_j^m(t_j+\frac{\beta}{m}\tau)=f(t_j+\frac{\beta}{m}\tau)$ for $0\le \beta \le m$. Using Lagrange polynomials, one can prove that there exists weights $(\omega_\beta^{(m)})_{0\le\beta\le m}$ such that
\begin{equation}
\int_{t_j}^{t_{j+1}}P_j^m(t)\drm t=\sum_{\beta=0}^m\omega_\beta^{(m)}f(t_{j,\beta}^{(m)})
\end{equation}
for any $0\le j\le J$. Then we get
\begin{equation}
\Big|\int_0^Tf(t)\drm t-\sum_{j=0}^{J-1}\sum_{\beta=0}^m\omega_{\beta}^{(m)}f(t_{j,\beta}^{(m)})\Big|\lesssim\tau^{m+2}
\end{equation}
for smooth functions $f$ and even integer $m\ge2$, see for example Chapter $6$ from \cite{Hoffman2001}. Of course, the above error depends a priori on the final time $T$, so in our case we need to carefully exploit the dispersive properties of the numerical scheme to eliminate this dependence.

\medskip

To obtain an error of order $\tau^{N-n}$ for $U_n$, a natural idea is to consider a Newton-Cotes quadrature formula of order $m=N-n-2$ for the time integral in \eqref{eq_Un_integrated}. Since this requires to work on a grid that depends on $m$, one needs to have the discrete lower order terms on the same grid. This imposes a strong condition on the discretization and we consider dyadic partitions of each interval~$[t_j,t_{j+1}]$. Due to its structure, this method yields numerical schemes of order at most $N\le4$ as higher-order accuracy is obstructed by error propagation. Nonetheless, we include this scheme here, as it remains significantly simpler than the one introduced next, while still achieving fourth-order accuracy, which is sufficient for the physical applications we consider. We present the case $N=4$, noting that lower-order cases $N<4$ can be straightforwardly obtained by adjusting the Newton–Cotes rule used for each term. Moreover, our other scheme \eqref{NTS} is only applicable in dimension 3 (so for the cubic case $p=3$) up to order $N=3$, whereas the \eqref{NQS} scheme allows for higher-order accuracy (namely $N=4$) in this setting.
 
We now define the \eqref{NQS} scheme as follows: for any $0 \leq j \leq J$, 
\begin{equation} \label{NQS} \tag{NQS}
\left|
\begin{aligned}
\ \mathfrak{U}_3^j& \coloneqq -i\sum_{\alpha=0}^{j-1}S_\tau(t_j-t_\alpha) \mathfrak{F}_3^{\alpha}\\
   & \text{where} \quad  \mathfrak{F}_3^{\alpha} \coloneqq  \sum_{n_1+\ldots+n_p=2} \mathfrak{U}_{n_1}^{\alpha} \overline{\mathfrak{U}_{n_2}^{\alpha}} \ldots \mathfrak{U}_{n_p}^{\alpha}, \\
\mathfrak{U}_2^j& \coloneqq -i\sum_{\alpha=0}^{j-1}\frac{\tau}{2} \left(S_\tau(t_j-t_\alpha) \mathfrak{F}_2^{\alpha} +  S_\tau(t_j-t_{\alpha+1}) \mathfrak{F}_2^{\alpha+1} \right) \\
& =- i \tau \sum_{\alpha=0}^{j-1}\sum_{\beta=0}^1 S_{\tau}(t_j-t_{\alpha,\beta}^{(1)}) (\omega_{\beta}^{(1)} \mathfrak{F}_2^{\alpha+\beta} )  \\
& \text{where} \quad \mathfrak{F}_2^{\alpha} \coloneqq \sum_{n_1+\ldots+n_p=1} \mathfrak{U}_{n_1}^{\alpha} \overline{\mathfrak{U}_{n_2}^{\alpha}} \ldots \mathfrak{U}_{n_p}^{\alpha} = \left(\frac{p+1}{2}|\mathfrak{U}_0^\alpha|^{p-1} \mathfrak{U}_1^\alpha+\frac{p-1}{2}|\mathfrak{U}_0^\alpha|^{p-3}(\mathfrak{U}_0^\alpha)^2 \overline{\mathfrak{U}_1^\alpha}\right)\\
\mathfrak{U}_1^j& \coloneqq  -i \sum_{\alpha=0}^{j-1} \tau \left( \frac{1}{6} S_{\tau}(t_j-t_{\alpha}) \mathfrak{F}_1^{\alpha} + \frac{2}{3} S_{\tau}(t_j-t_{\alpha+\frac12}) \mathfrak{F}_1^{\alpha+\frac12} + \frac{1}{6} S_{\tau}(t_j-t_{\alpha+1}) \mathfrak{F}_1^{\alpha+1} \right) \\
& =- i \tau \sum_{\alpha=0}^{j-1}\sum_{\beta=0}^2 S_{\tau}(t_j-t_{\alpha,\beta}^{(2)}) (\omega_{\beta}^{(2)} \mathfrak{F}_1^{\alpha+\frac{\beta}{2}} )  \\
& \text{where} \quad \mathfrak{F}_1^{\alpha} \coloneqq \sum_{n_1+\ldots+n_p=0} \mathfrak{U}_{n_1}^{\alpha} \overline{\mathfrak{U}_{n_2}^{\alpha}} \ldots \mathfrak{U}_{n_p}^{\alpha} = |\mathfrak{U}_0^{\alpha}|^{p-1} \mathfrak{U}_0^{\alpha}, \\
 \ \mathfrak{U}_0^{\alpha+\frac{1}{2}}& \coloneqq S_{\tau} (\tau/2 ) \mathfrak{U}_0^{\alpha} \ \text{for all} \ \alpha \in \frac{\tau}{2}\IZ, \quad \text{with} \quad \mathfrak{U}_0^0=\varphi,
\end{aligned}
\right.
\end{equation}
where we have respectively used a left rectangle rule to discretize $U_3$, a trapezoidal rule for $U_2$ and and Simpson’s rule for $U_1$ (which requires the discretization of $U_0$ on a finer grid). We now state our first convergence result, the proof of which will be provided in Section~\ref{sec:convergence_NQS}.

\medskip

\begin{theorem}\label{theorem_convergence_NQS}
For $1 \leq N \leq 4$ and $\varphi\in\Sigma\cap H^{2N}(\IR^d)$, let $u$ be the solution to equation \eqref{NLS} with initial data $u(0)=\varphi$. Fix $T,\tau>0$ and $J\in\IN$ such that $T=J\tau$ and consider the \eqref{NQS} numerical scheme $(\mathfrak{U}_n^j)_{n,j}$ for $0\le n\le N-1$ and $0\le j\le J$ defined previously. Then there exists a constant~$C=C(N,d,\|\varphi\|_{\Sigma},\|\varphi\|_{H^{2N}})>0$ independent of $T$ such that
\begin{equation}
\sup_{0\le j\le J}\|u(t_j)-\sum_{n=0}^{N-1}\eps^n \mathfrak{U}_n^j\|_{L^2_x}\le C\sum_{n=0}^{N}\eps^n\tau^{N-n}.
\end{equation}
In particular, we get
\begin{equation}
\sup_{0\le j\le J}\|u(t_j)-\sum_{n=0}^{N-1}\eps^n \mathfrak{U}_n^j\|_{L^2_x}\le C\eps^N
\end{equation}
for $\tau\le\eps$.
\end{theorem}

\begin{remark}
One could, in principle, define an analogous scheme for arbitrarily high order. However, such schemes do not converge a priori, as it becomes impossible to control the local error
\[ U_n^{\tau} \left(t_j+\frac{1}{2} \right)-\mathfrak{U}_n^{j+\frac{1}{2}}\]
for $n\ge1$. This issue comes from the fact that a Newton–Cotes quadrature of order $N\ge1$ typically degrades to first order if even a single point is removed from the discretization. While one might hope to use the equation itself to infer values on a coarser grid from already constructed approximations, such a strategy appears unfeasible in this context. This highlights the sensitivity of high-order schemes to local errors. Similar difficulties also arise with more traditional exponential integrators, such as Runge–Kutta methods (which may fail to satisfy dispersive estimates) or Lawson-type methods. Moreover, our approach relies crucially on discrete Strichartz estimates, which are not known to hold on non-uniform temporal grids.
\end{remark}

\subsection{Nested Taylor Scheme}\label{subsec:NTS}

We now construct the Nested Taylor Scheme \eqref{NTS} based on a high-order Taylor expansion in time. Rather than introducing a finer temporal grid, we discretize the time derivatives of the solution directly. Exploiting the underlying equation, we can recursively propagate the error, enabling the construction of arbitrarily high-order methods. To approximate $U_n$ for $n\ge1$, we apply the Taylor formula with an explicit remainder to the nonlinear term. In the following, we first fix a $p$-uplet $(n_1,\ldots,n_p)\in\llbracket0,N-1\rrbracket^p$ such that $n_1+\ldots+n_p=n-1$ and consider
\begin{align}
\int_0^{t_j}S(t_j-s)&U_{n_1}(s)\overline{U_{n_2}(s)}\ldots U_{n_p}(s)\drm s=\sum_{\alpha=0}^{j-1}\int_{t_\alpha}^{t_{\alpha+1}}S(t_j-s)U_{n_1}(s)\overline{U_{n_2}(s)}\ldots U_{n_p}(s)\drm s\\
&=\sum_{\alpha=0}^{j-1}\sum_{\beta=0}^{m_n}\int_{t_\alpha}^{t_{\alpha+1}}\frac{(s-t_\alpha)^\beta}{\beta!}\partial_s^\beta\big(S(t_j-s)U_{n_1}(s)\overline{U_{n_2}(s)}\ldots U_{n_p}(s)\big)(t_\alpha)\drm s\\
&\quad+\sum_{\alpha=0}^{j-1}\int_{t_\alpha}^{t_{\alpha+1}}\int_{t_\alpha}^{s_1}\frac{(s_1-s_2)^{m_n}}{m_n!}\partial_s^{\beta+1}\big(S(t_j-s)U_{n_1}(s)\overline{U_{n_2}(s)}\ldots U_{n_p}(s)\big)(s_2)\drm s_2\drm s_1
\end{align}
where the order $m_n$ is defined by
\begin{equation}
m_n\coloneqq N-n-1
\end{equation} 
decreases as $n$ grows from $0$ to $N-1$. Indeed, the smaller $n$ is, the higher the required accuracy, that is an error of order $\tau^{N-n}$. Using Leibniz rule, the time derivative of order $\beta$ of~$F_n$ can be expressed as a combination of time derivatives of $U_{n'}$ (with $n'<n$) up to order $\beta$. Since each~$U_n$ satisfies a linear Schrödinger equation driven by lower-order terms, we can recursively use the equation to convert time derivatives into spatial derivatives and products involving lower-order components. Unlike the previous method, this allows us to propagate fine control on the quantities~$\nabla^k U_n$ which themselves solve linear Schrödinger equations, making it possible to design numerical schemes of arbitrarily high order. Following the approach introduced by \textsc{Butcher} in his seminal work \cite{Butcher1972} on high-order Runge–Kutta methods, we now introduce the decorated tree notation, which provides the formal framework for defining our high-order scheme. Consider
\begin{equation}
U_0=\begin{tikzpicture}[baseline]
\node (1) at (0,0) [inner sep=0pt] {};
\node (2) at (0,3/8) [tdot] {};
\draw (1) -- (2);
\end{tikzpicture}
\end{equation}
the free propagation of the initial data $\varphi$. Here the dot represents the initial data while the edge stands for the Schrödinger propagator. When considering a conjugation, we will use dotted edge with
\begin{equation}
\overline{U_0}=\begin{tikzpicture}[baseline]
\node (2) at (0,3/8) [tdot] {};
\draw[densely dotted] (0,0) -- (2);
\end{tikzpicture}
\end{equation}
and for product we just link trees thus for instance for cubic interactions
\begin{equation}
|U_0|^2U_0=\begin{tikzpicture}[baseline]
\node (2) at (0,3/8) [tdot] {};
\node (3) at (-1/4,3/8) [tdot] {};
\node (4) at (1/4,3/8) [tdot] {};
\draw[densely dotted] (0,0) -- (2);
\draw (0,0) -- (3);
\draw (0,0) -- (4);
\end{tikzpicture}
\end{equation}
where the trees are not planar thus the choice of the position of the dotted line is not important. Since we will need to compute spatial derivative $\nabla^k$ which commute with the free propagator, we adopt the notation
\begin{equation}
\nabla^kU_0=\begin{tikzpicture}[baseline]
\node (2) at (0,3/8) [tdot] {};
\node at (1/6,3/16) [] {\tiny$k$};
\draw (0,0) -- (2);
\end{tikzpicture}
\end{equation}
where $k=0$ and no index denote the same quantity with analog notation for the conjugate. In the following, we shall call decorations the fact that a line is dotted and that there is an integer $k$ on edges. A tree without decoration is called a \textit{bare tree}. Note that since a black node represents~$\varphi$, a tree with $q$ leaves is a $q$-linear functional of the initial data. We fix $p=3$ in the following examples of trees to keep the notation lighter while still explaining the general case. The recursive definition
\begin{equation}
(i\partial_t-\Delta)U_n=\sum_{n_1+n_2+n_3=n-1}U_{n_1}\overline{U_{n_2}}U_{n_3}
\end{equation}
comes done to adding new trees in our collection. Since each $U_n$ is defined first via an integration~$(i\partial_t-\Delta)^{-1}$ applied to a product of previously constructed terms, they are represented with a planted tree of the form
\begin{equation}
a=\begin{tikzpicture}[baseline]
\node (2) at (0,1/2) [] {b};
\draw (0,0) -- (2);
\end{tikzpicture}
\end{equation}
where $b$ is a product of $p$ trees. In the following we denote by $a^t=b$ a planted tree without its first unique edge. This gives the construction rule
\begin{equation}\label{eq:recursivetree}
a=\begin{tikzpicture}[baseline]
\node (2) at (0,1/2+6/8) [] {$a_2^t$};
\node (3) at (-1/2,1/2+6/8) [] {$a_1^t$};
\node (4) at (1/2,1/2+6/8) [] {$a_3^t$};
\draw[densely dotted] (0,1/2) -- (2);
\draw (0,1/2) -- (3);
\draw (0,1/2) -- (4);
\draw (0,1/2) -- (0,0);
\end{tikzpicture}
\end{equation}
with trees $a_1,a_2,a_3$ previously constructed. In particular, this provides a very efficient way to represent $F_n$ as a sum of $p$ rooted trees with $1+(p-1)n$ leaves for general $p\ge1$. 

\medskip 

\begin{definition}
For any integer $n\ge1$, we define $\CS_n$ as the set of planted trees with $p+(p-1)(n-1)$ leaves and $\CT_n$ the set of decorated trees from $\CT_n$ where each nodes except the root has $p$ upgoing edges among which exactly $\frac{p-1}{2}$ are dotted.
\end{definition}

\medskip 

With this new notation, we can then write that
\begin{equation}
U_n=\sum_{a\in\CT_n}c(a)a
\end{equation}
with $c(a)\in\IN$ coefficients coming from the symetries of the trees that we do not carefuly track here. In particular, the number of leaves of a tree determines its index $n$. For $p=3$, we have
\begin{equation}
\CT_1=\Big\{\ \begin{tikzpicture}[baseline]
\node (2) at (0,-1/4+6/8) [tdot] {};
\node (3) at (-1/4,-1/4+6/8) [tdot] {};
\node (4) at (1/4,-1/4+6/8) [tdot] {};
\draw[densely dotted] (0,0) -- (2);
\draw (0,0) -- (3);
\draw (0,0) -- (4);
\draw (0,0) -- (0,-1/4);
\end{tikzpicture}\ \Big\}
\quad\text{and}\quad
\CT_2=\Big\{\ \begin{tikzpicture}[baseline]
\node (2) at (-1/4,1/4+6/8) [tdot] {};
\node (3) at (-1/2,1/4+6/8) [tdot] {};
\node (4) at (0,1/4+6/8) [tdot] {};
\node (5) at (-1/2,1/2) [tdot] {};
\node (6) at (-3/4,1/2) [tdot] {};
\draw[densely dotted] (-1/4,1/2) -- (2);
\draw (-1/4,1/2) -- (3);
\draw (-1/4,1/2) -- (4);
\draw (-1/4,1/2) -- (-1/2,0);
\draw[densely dotted] (5) -- (-1/2,0);
\draw (6) -- (-1/2,0);
\draw (-1/2,-1/4) -- (-1/2,0);
\end{tikzpicture}
\quad,
\begin{tikzpicture}[baseline]
\node (2) at (-1/4,1/4+6/8) [tdot] {};
\node (3) at (-1/2,1/4+6/8) [tdot] {};
\node (4) at (0,1/4+6/8) [tdot] {};
\node (5) at (-1/2,1/2) [tdot] {};
\node (6) at (-3/4,1/2) [tdot] {};
\draw[densely dotted] (-1/4,1/2) -- (2);
\draw (-1/4,1/2) -- (3);
\draw (-1/4,1/2) -- (4);
\draw[densely dotted] (-1/4,1/2) -- (-1/2,0);
\draw (5) -- (-1/2,0);
\draw (6) -- (-1/2,0);
\draw (-1/2,-1/4) -- (-1/2,0);
\end{tikzpicture}\ \Big\}
\end{equation}
where the first tree of $\CT_2$ has coefficient $2$ since there are two full edges where one can graft the tree of $\CT_1$. For $\CT_3$, one has to consider the growing mechanism given by \eqref{eq:recursivetree} with either two trees from $\CT_1$ and one leaf, or one tree from $\CT_2$ with two leaves. We then get
\begin{equation}
\CT_3=\Big\{\ \begin{tikzpicture}[baseline]
\node (2) at (-1/4+1/10,1/4+6/8) [tdot] {};
\node (3) at (-1/2+1/10,1/4+6/8) [tdot] {};
\node (4) at (0+1/10,1/4+6/8) [tdot] {};
\node (5) at (-1/2,1/2) [tdot] {};
\draw[densely dotted] (-1/4+1/10,1/2) -- (2);
\draw (-1/4+1/10,1/2) -- (3);
\draw (-1/4+1/10,1/2) -- (4);
\draw (-1/4+1/10,1/2) -- (-1/2,0);
\draw[densely dotted] (5) -- (-1/2,0);
\draw (-3/4-1/10,1/2) -- (-1/2,0);
\node (6) at (-1-1/10,1/4+6/8) [tdot] {};
\node (7) at (-1/2-1/10,1/4+6/8) [tdot] {};
\node (8) at (-3/4-1/10,1/4+6/8) [tdot] {};
\draw[densely dotted] (-3/4-1/10,1/2) -- (6);
\draw (-3/4-1/10,1/2) -- (7);
\draw (-3/4-1/10,1/2) -- (8);
\draw (-1/2,0) -- (-1/2,-1/4);
\end{tikzpicture}
\ ,
\begin{tikzpicture}[baseline]
\node (2) at (-1/4+1/10,1/4+6/8) [tdot] {};
\node (3) at (-1/2+1/10,1/4+6/8) [tdot] {};
\node (4) at (0+1/10,1/4+6/8) [tdot] {};
\node (5) at (-1/2,1/2) [tdot] {};
\draw[densely dotted] (-1/4+1/10,1/2) -- (2);
\draw (-1/4+1/10,1/2) -- (3);
\draw (-1/4+1/10,1/2) -- (4);
\draw (-1/4+1/10,1/2) -- (-1/2,0);
\draw (5) -- (-1/2,0);
\draw[densely dotted] (-3/4-1/10,1/2) -- (-1/2,0);
\node (6) at (-1-1/10,1/4+6/8) [tdot] {};
\node (7) at (-1/2-1/10,1/4+6/8) [tdot] {};
\node (8) at (-3/4-1/10,1/4+6/8) [tdot] {};
\draw[densely dotted] (-3/4-1/10,1/2) -- (6);
\draw (-3/4-1/10,1/2) -- (7);
\draw (-3/4-1/10,1/2) -- (8);
\draw (-1/2,0) -- (-1/2,-1/4);
\end{tikzpicture}
\ ,\ 
\begin{tikzpicture}[baseline]
\node (2) at (-1/4,1/4+6/8) [tdot] {};
\node (3) at (-1/2,1/4+6/8) [tdot] {};
\node (4) at (0,1/4+6/8) [tdot] {};
\node (5) at (-1/2,1/2) [tdot] {};
\node (6) at (-3/4,1/2) [tdot] {};
\draw[densely dotted] (-1/4,1/2) -- (2);
\draw (-1/4,1/2) -- (3);
\draw (-1/4,1/2) -- (4);
\draw (-1/4,1/2) -- (-1/2,0);
\draw[densely dotted] (5) -- (-1/2,0);
\draw (6) -- (-1/2,0);
\draw (-1/2,0) -- (-3/4,-1/2);
\node (7) at (-3/4,0) [tdot] {};
\node (8) at (-1,0) [tdot] {};
\draw[densely dotted] (7) -- (-3/4,-1/2);
\draw (8) -- (-3/4,-1/2);
\draw (-3/4,-3/4) -- (-3/4,-1/2);
\end{tikzpicture}
\ ,\ 
\begin{tikzpicture}[baseline]
\node (2) at (-1/4,1/4+6/8) [tdot] {};
\node (3) at (-1/2,1/4+6/8) [tdot] {};
\node (4) at (0,1/4+6/8) [tdot] {};
\node (5) at (-1/2,1/2) [tdot] {};
\node (6) at (-3/4,1/2) [tdot] {};
\draw[densely dotted] (-1/4,1/2) -- (2);
\draw (-1/4,1/2) -- (3);
\draw (-1/4,1/2) -- (4);
\draw (-1/4,1/2) -- (-1/2,0);
\draw[densely dotted] (5) -- (-1/2,0);
\draw (6) -- (-1/2,0);
\draw[densely dotted] (-1/2,0) -- (-3/4,-1/2);
\node (7) at (-3/4,0) [tdot] {};
\node (8) at (-1,0) [tdot] {};
\draw (7) -- (-3/4,-1/2);
\draw (8) -- (-3/4,-1/2);
\draw (-3/4,-3/4) -- (-3/4,-1/2);
\end{tikzpicture}
\ ,\ 
\begin{tikzpicture}[baseline]
\node (2) at (-1/4,1/4+6/8) [tdot] {};
\node (3) at (-1/2,1/4+6/8) [tdot] {};
\node (4) at (0,1/4+6/8) [tdot] {};
\node (5) at (-1/2,1/2) [tdot] {};
\node (6) at (-3/4,1/2) [tdot] {};
\draw[densely dotted] (-1/4,1/2) -- (2);
\draw (-1/4,1/2) -- (3);
\draw (-1/4,1/2) -- (4);
\draw[densely dotted] (-1/4,1/2) -- (-1/2,0);
\draw (5) -- (-1/2,0);
\draw (6) -- (-1/2,0);
\draw (-1/2,0) -- (-3/4,-1/2);
\node (7) at (-3/4,0) [tdot] {};
\node (8) at (-1,0) [tdot] {};
\draw[densely dotted] (7) -- (-3/4,-1/2);
\draw (8) -- (-3/4,-1/2);
\draw (-3/4,-3/4) -- (-3/4,-1/2);
\end{tikzpicture}
\ ,\ 
\begin{tikzpicture}[baseline]
\node (2) at (-1/4,1/4+6/8) [tdot] {};
\node (3) at (-1/2,1/4+6/8) [tdot] {};
\node (4) at (0,1/4+6/8) [tdot] {};
\node (5) at (-1/2,1/2) [tdot] {};
\node (6) at (-3/4,1/2) [tdot] {};
\draw[densely dotted] (-1/4,1/2) -- (2);
\draw (-1/4,1/2) -- (3);
\draw (-1/4,1/2) -- (4);
\draw[densely dotted] (-1/4,1/2) -- (-1/2,0);
\draw (5) -- (-1/2,0);
\draw (6) -- (-1/2,0);
\draw[densely dotted] (-1/2,0) -- (-3/4,-1/2);
\node (7) at (-3/4,0) [tdot] {};
\node (8) at (-1,0) [tdot] {};
\draw (7) -- (-3/4,-1/2);
\draw (8) -- (-3/4,-1/2);
\draw (-3/4,-3/4) -- (-3/4,-1/2);
\end{tikzpicture}\ \Big\}
\end{equation}
where one can observe two distinct tree structures, each corresponding to a different configuration of dotted edges. As with the example of $\CT_2$, the associated coefficients can be computed by counting the number of graftings that produce a given tree, though we do not provide the general formula here. Since $a^t$ denotes the tree $a$ with its root removed, we can also write that
\begin{equation}
F_n=\sum_{a\in\CT_n}c(a)a^t
\end{equation}
with the same coefficients and the convention $a^t=0$ for $a=\begin{tikzpicture}[baseline]
\node (1) at (0,0) [inner sep=0pt] {};
\node (2) at (0,3/8) [tdot] {};
\draw (1) -- (2);
\end{tikzpicture}$ . Thanks to the recursive construction from \eqref{eq:recursivetree}, each $a^t$ for $a\in\CT_n$ is a product of $p$ trees 
\begin{equation}
a^t=b_1\ldots b_p
\end{equation}
with $b_k\in\{a_k,\overline{a_k}\}$ and $a_k\in \CT_{n_k}$ such that $n_1+\ldots+n_p=n-1$ where exactly $\frac{p-1}{2}$ are conjugated. Then Taylor expansions of $U_n$ requires the computation of
\begin{align}
\partial_s\big(S(t-s) b_1\ldots b_p\big)&=\sum_{k=1}^p S(t-s)(\partial_s-i\Delta)b_k\prod_{k'\neq k}b_{k'}\\ 
&\quad-2i\sum_{1\le k_1<k_2\le p}S(t-s)\nabla b_{k_1}\cdot\nabla b_{k_2}\prod_{k'\neq k_1,k_2}b_{k'}
\end{align}
using the Leibniz rule for the Laplacian $\Delta$ which arises from the identity $\partial_s S=-i\Delta$, and similarly for higher-order time derivatives. Each tree $a_k$ being rooted, the gradient term amounts to adding decoration $k=1$ on its first vertical edge while the conjugation requires to have a dotted first edge. For the first term, we use the equation to convert time derivative into spatial derivative. The equation on $U_n$ gives
\begin{equation}
(\partial_t-i\Delta)U_n=-iF_n
\end{equation}
while for the conjugate, we have
\begin{equation}
(\partial_t-i\Delta)\overline{U_n}=iF_n-2i\Delta\overline{U_n}
\end{equation}
which involves the second order derivative as an extra term. In the end, taking the time derivative involves new terms obtained from $(a_1,\ldots,a_p)$ by applying the following three rules:
\begin{enumerate}
    \item[\textbf{(A)}] Add an index $1$ at the root of two different trees $b_i$ and $b_j$.
    \item[\textbf{(B)}] Add an index $2$ at the root of a conjugated tree $b_i=\overline{a_i}$.
    \item[\textbf{(C)}] Cut the edge root of a tree $b_i$ and propagate the conjugation to the following edge if necessary.
\end{enumerate}
In particular, note that only the edges of the first floor can have integer decoration. If we omit the decoration, the bare trees of $\CT_1,\CT_2$ and $\CT_3$ are simply given by
\begin{equation}
\begin{tikzpicture}[baseline]
\node (2) at (0,1/4+6/8) [tdot] {};
\node (3) at (-1/4,1/4+6/8) [tdot] {};
\node (4) at (1/4,1/4+6/8) [tdot] {};
\draw (0,1/2) -- (2);
\draw (0,1/2) -- (3);
\draw (0,1/2) -- (4);
\draw (0,1/2) -- (0,1/4);
\end{tikzpicture}
\quad
\begin{tikzpicture}[baseline]
\node (2) at (-1/4,1/4+6/8) [tdot] {};
\node (3) at (-1/2,1/4+6/8) [tdot] {};
\node (4) at (0,1/4+6/8) [tdot] {};
\node (5) at (-1/2,1/2) [tdot] {};
\node (6) at (-3/4,1/2) [tdot] {};
\draw (-1/4,1/2) -- (2);
\draw (-1/4,1/2) -- (3);
\draw (-1/4,1/2) -- (4);
\draw (-1/4,1/2) -- (-1/2,0);
\draw (5) -- (-1/2,0);
\draw (6) -- (-1/2,0);
\draw (-1/2,-1/4) -- (-1/2,0);
\end{tikzpicture}
\quad
\begin{tikzpicture}[baseline]
\node (2) at (-1/4+1/10,1/4+6/8) [tdot] {};
\node (3) at (-1/2+1/10,1/4+6/8) [tdot] {};
\node (4) at (0+1/10,1/4+6/8) [tdot] {};
\node (5) at (-1/2,1/2) [tdot] {};
\draw (-1/4+1/10,1/2) -- (2);
\draw (-1/4+1/10,1/2) -- (3);
\draw (-1/4+1/10,1/2) -- (4);
\draw (-1/4+1/10,1/2) -- (-1/2,0);
\draw (5) -- (-1/2,0);
\draw (-3/4-1/10,1/2) -- (-1/2,0);
\node (6) at (-1-1/10,1/4+6/8) [tdot] {};
\node (7) at (-1/2-1/10,1/4+6/8) [tdot] {};
\node (8) at (-3/4-1/10,1/4+6/8) [tdot] {};
\draw (-3/4-1/10,1/2) -- (6);
\draw (-3/4-1/10,1/2) -- (7);
\draw (-3/4-1/10,1/2) -- (8);
\draw (-1/2,0) -- (-1/2,-1/4);
\end{tikzpicture}
\quad
\begin{tikzpicture}[baseline]
\node (2) at (-1/4,1/4+6/8) [tdot] {};
\node (3) at (-1/2,1/4+6/8) [tdot] {};
\node (4) at (0,1/4+6/8) [tdot] {};
\node (5) at (-1/2,1/2) [tdot] {};
\node (6) at (-3/4,1/2) [tdot] {};
\draw (-1/4,1/2) -- (2);
\draw (-1/4,1/2) -- (3);
\draw (-1/4,1/2) -- (4);
\draw (-1/4,1/2) -- (-1/2,0);
\draw (5) -- (-1/2,0);
\draw (6) -- (-1/2,0);
\draw (-1/2,0) -- (-3/4,-1/2);
\node (7) at (-3/4,0) [tdot] {};
\node (8) at (-1,0) [tdot] {};
\draw (7) -- (-3/4,-1/2);
\draw (8) -- (-3/4,-1/2);
\draw (-3/4,-3/4) -- (-3/4,-1/2);
\end{tikzpicture}
\end{equation}
where $a_i\in\CT_i$ has $3+2(i-1)$ leaves for $p=3$. While rules \textbf{(A)} and \textbf{(B)} only affect the decoration, the rule \textbf{(C)} acts on the structure of the tree. 

\medskip

\begin{definition}
For $\beta\ge0$, we define $\CS_n^\beta$ as the set $\CS_n$ enhanced with planted trees obtained by $\beta$ applications of the rule \textbf{(C)} to its elements. We also define $\CT_n^\beta$ as the set of decorated trees $\CT_n$ enhanced with trees obtained by $\beta$ applications of the rules \textbf{(A)},\textbf{(B)} and \textbf{(C)} to its elements.
\end{definition}

\medskip

By construction, the sets $\CS_n^\beta$ are increasing with respect to $\beta$. Applying rule \textbf{(C)} to a tree decreases the sum of the length in each direction of the $p$ edges of the first floor. Consequently, repeated application of this rule eventually yields the minimal-height tree of height two, hence
\begin{equation}
\CS_n^\infty \coloneqq \bigcup_{\beta\ge0}\CS_n^\beta
\end{equation}
is a finite set. Again for $p=3$, we have for example $\CS_1^\beta=\CS_1$ and
\begin{equation}
\CS_2^\beta=\CS_2\cup\Big\{\ \begin{tikzpicture}[baseline]
\node (2) at (-1/4,1/2) [tdot] {};
\node (3) at (-1/2,1/2) [tdot] {};
\node (4) at (0,1/2) [tdot] {};
\node (5) at (1/2,1/2) [tdot] {};
\node (6) at (1/4,1/2) [tdot] {};
\draw (0,0) -- (2);
\draw (0,0) -- (3);
\draw (0,0) -- (4);
\draw (0,0) -- (5);
\draw (0,0) -- (6);
\draw (0,0) -- (0,-1/4);
\end{tikzpicture}\ \Big\}
\end{equation}
for any $\beta\ge1$. We also have
\begin{equation}
\CS_3^1=\CS_3\cup\Big\{\ \begin{tikzpicture}[baseline]
\node (2) at (-1/4+1/10,1/4+6/8) [tdot] {};
\node (3) at (-1/2+1/10,1/4+6/8) [tdot] {};
\node (4) at (0+1/10,1/4+6/8) [tdot] {};
\node (5) at (-1/2,1/2) [tdot] {};
\draw (-1/4+1/10,1/2) -- (2);
\draw (-1/4+1/10,1/2) -- (3);
\draw (-1/4+1/10,1/2) -- (4);
\draw (-1/4+1/10,1/2) -- (-1/2,0);
\draw (5) -- (-1/2,0);
\node (6) at (-1-1/4,1/2) [tdot] {};
\node (7) at (-1/2-1/4,1/2) [tdot] {};
\node (8) at (-3/4-1/4,1/2) [tdot] {};
\draw (-1/2,0) -- (6);
\draw (-1/2,0) -- (7);
\draw (-1/2,0) -- (8);
\draw (-1/2,0) -- (-1/2,-1/4);
\end{tikzpicture}\ \Big\}
\end{equation}
and
\begin{equation}
\CS_3^\beta=\CS_3^1\cup\Big\{\ \begin{tikzpicture}[baseline]
\node (2) at (-1/4,1/2) [tdot] {};
\node (3) at (-1/2,1/2) [tdot] {};
\node (4) at (0,1/2) [tdot] {};
\node (5) at (1/2,1/2) [tdot] {};
\node (6) at (1/4,1/2) [tdot] {};
\node (7) at (3/4,1/2) [tdot] {};
\node (8) at (-3/4,1/2) [tdot] {};
\draw (0,0) -- (2);
\draw (0,0) -- (3);
\draw (0,0) -- (4);
\draw (0,0) -- (5);
\draw (0,0) -- (6);
\draw (0,0) -- (7);
\draw (0,0) -- (8);
\draw (0,0) -- (0,-1/4);
\end{tikzpicture}\ \Big\}
\end{equation}
for any $\beta\ge2$. Any decorated tree $a\in\CT_n^\beta$ comes from a bare tree in $\CS_n^\beta$ according to the following rules. At the first floor, there is exactly $p+q(p-1)$ upgoing edges for an integer $q\ge0$ since applying rule \textbf{(C)} adds $(p-1)$ edges. The first rule is to have exactly $(1+q)\frac{p-1}{2}$ dotted edges at the first floor and $\frac{p-1}{2}$ dotted edges for each internal nodes. Then one adds weights on the edges of the first floor corresponding to the application of rules \textbf{(A)} and \textbf{(B)}. Since rules \textbf{(A)} and \textbf{(B)} each increase the total weight by 2, and the number of times rule \textbf{(C)} is applied to construct $a \in \CT_n^{\beta}$ is entirely determined by the integer $q$, which in turn depends only on the underlying bare tree structure in~$\CS_n^{\beta}$,the resulting condition on the weights becomes
\begin{equation}\label{eq:TreeTotalWeight}
\frac{1}{2}\sum_{e\in\CE_1(a)}k_e=\beta-q
\end{equation}
where $\CE_1(a)$ denotes the set of edges of the first floor and $k_e$ the weight of an edge. For the first new structure in $\CS_3^1$ in the case $p=3$, we get for example
\begin{equation}
\begin{tikzpicture}[baseline]
\node (2) at (-1/4+1/10,1/4+6/8) [tdot] {};
\node (3) at (-1/2+1/10,1/4+6/8) [tdot] {};
\node (4) at (0+1/10,1/4+6/8) [tdot] {};
\node (5) at (-1/2,1/2) [tdot] {};
\draw[densely dotted] (-1/4+1/10,1/2) -- (2);
\draw (-1/4+1/10,1/2) -- (3);
\draw (-1/4+1/10,1/2) -- (4);
\draw[densely dotted] (-1/4+1/10,1/2) -- (-1/2,0);
\draw[densely dotted] (5) -- (-1/2,0);
\node (6) at (-1-1/4,1/2) [tdot] {};
\node (7) at (-1/2-1/4,1/2) [tdot] {};
\node (8) at (-3/4-1/4,1/2) [tdot] {};
\draw (-1/2,0) -- (6);
\draw (-1/2,0) -- (7);
\draw (-1/2,0) -- (8);
\draw (-1/2,0) -- (-1/2,-1/4);
\end{tikzpicture}
\end{equation}
for the dotted edges on which it only remains to add weight. In this example, there is $5$ upgoings edges at the first floor thus $q=1$ hence the rule \textbf{(C)} has been applied once. The total weight has to be equal to $2(\beta-1)$, so for example the decorated tree
\begin{equation}
\begin{tikzpicture}[baseline]
\node (2) at (-1/4+1/10,1/4+6/8) [tdot] {};
\node (3) at (-1/2+1/10,1/4+6/8) [tdot] {};
\node (4) at (0+1/10,1/4+6/8) [tdot] {};
\node (5) at (-1/2,1/2) [tdot] {};
\draw[densely dotted] (-1/4+1/10,1/2) -- (2);
\draw (-1/4+1/10,1/2) -- (3);
\draw (-1/4+1/10,1/2) -- (4);
\draw[densely dotted] (-1/4+1/10,1/2) -- (-1/2,0);
\draw[densely dotted] (5) -- (-1/2,0);
\node (6) at (-1-1/4,1/2) [tdot] {};
\node (7) at (-1/2-1/4,1/2) [tdot] {};
\node (8) at (-3/4-1/4,1/2) [tdot] {};
\draw (-1/2,0) -- (6);
\draw (-1/2,0) -- (7);
\draw (-1/2,0) -- (8);
\node at (-1-1/10,1/4) [] {\tiny$1$};
\node at (-1/2+1/15,1/4) [] {\tiny$1$};
\node at (-1/6,1/4) [] {\tiny$2$};
\draw (-1/2,0) -- (-1/2,-1/4);
\end{tikzpicture}
\end{equation}
belongs to $\CT_3^3$. In particular, given any decorated tree $a\in\CT_n^\beta$, one can compute the parameter $n$ from the number of leaves and $\beta$ from the total weight. With these notations, we write
\begin{equation}
\partial_s^\beta\big(S(t-s) F_n(s)\big)=\sum_{a\in\CT_n^\beta}c(a)S(t-s) a^t(s)
\end{equation}
where $c(a)\in\IC$ are now complex coefficients. The tree notation provides a compact and structured way to represent the numerous terms arising in high-order Taylor expansions. By construction, for any $a\in\CT_n^\beta$, $a^\cut$ is expressed as a product of rooted decorated trees, each carrying derivatives on their first edge, namely
\begin{equation}\label{eq:arbrecut}
a^\cut=\prod_{e=1}^{p+q(p-1)}\nabla^{k_e}b_e
    \end{equation}
where $q\ge0$ corresponds to possible applications of the rule \textbf{(C)}, $b_e\in\{a_e,\overline{a_e}\}$ with $a_e\in\CT_{n_e}$ with $n_1+\ldots+n_{p+q(p-1)}=n$ and $2(\beta-q)=k_1+\ldots+k_{p+q(p-1)}$. We then rewrite with such formalism
\begin{align}
U_n(t_j)&=-i\int_0^{t_j} S(t_j-s)F_n(s)\drm s\\
&=-i\sum_{\alpha=0}^{j-1}\int_{t_\alpha}^{t_{\alpha+1}} S(t_j-s) F_n(s)\drm s\\
&=-i\sum_{\alpha=0}^{j-1}\sum_{a\in\CT_n}c(a)\int_{t_\alpha}^{t_{\alpha+1}} S(t_j-s) a^t(s)\drm s\\ 
&=-i\sum_{\alpha=0}^{j-1}\sum_{\beta=0}^{m_n}\sum_{a\in\CT_n}c(a)\int_{t_\alpha}^{t_{\alpha+1}}\frac{(s-t_\alpha)^\beta}{\beta!}\partial_s^{\beta}\big(S(t_j-s) a^t(s)\big)(t_\alpha)\drm s\\ 
&\quad-i\sum_{a\in\CT_n}\sum_{\alpha=0}^{j-1}\sum_{\beta=0}^{m_n}c(a)\int_{t_\alpha}^{t_{\alpha+1}}\int_{t_\alpha}^{s_1}\frac{(s_1-s_2)^{m_n}}{(m_n)!}\partial_s^{\beta+1}\big(S(t_j-s) a^t(s)\big)(s_2)\drm s_2\drm s_1\\
&=-i\sum_{\alpha=0}^{j-1}\sum_{\beta=0}^{m_n}\sum_{a\in\CT_n^\beta}c(a)\int_{t_\alpha}^{t_{\alpha+1}}\frac{(s-t_\alpha)^\beta}{\beta!}S(t_j-t_\alpha)a^t(t_\alpha)\drm s\\ 
&\quad-i\sum_{\alpha=0}^{j-1}\sum_{\beta=0}^{m_n}\sum_{a\in\CT_n^{m_n+1}}c(a)\int_{t_\alpha}^{t_{\alpha+1}}\int_{t_\alpha}^{s_1}\frac{(s_1-s_2)^{m_n}}{(m_n)!}S(t_j-s_2)a^t(s_2)\drm s_2\drm s_1 
\end{align}
which yields an expression for $U_n(t_j)$ via the tree expansion evaluated on the discrete grid $(t_j)_{j}$, up to a remainder term. Disregarding this remainder, the approximation involves computing the triangular set of decorated trees $\CT_n^\beta$ with $0\le n\le N-1$ and $0\le\beta\le m_N=N-1-n$. Then each decorated tree $a\in\CT_n^\beta$ has to be approximated with an error of order $\tau^{N-n-\beta}$ using the factor $\tau^\beta$ from the Taylor expansion. Since one time derivatives costs at most two spatial derivatives, we need to approximate 
\begin{equation}\label{ConditionNestedErrors}
\nabla^kU_n\ \text{for }k\le 2(N-n-2)\text{ with an error of order }\tau^{N-n-1-\lceil\frac{k}{2}\rceil}
\end{equation}
for $0\le n\le N-1$. Indeed, a term $\nabla^kU_n$ appears in the approximation of a tree $a\in\CT_{n'}^\beta$ where $n'>n$ and $2\beta\ge k$ with an error of order $\tau^{N-n'-\beta}$ by construction. The choice of parameters that minimizes the required accuracy corresponds to $(n',\beta)=(n+1,\lceil\frac{k}{2}\rceil)$, leading to an error of order $\tau^{N-n-1-\lceil\frac{k}{2}\rceil}$. We thus define
\begin{equation}
m_n^k \coloneqq N-n-\Big\lceil\frac{k}{2}\Big\rceil-2
\end{equation}
for $0\le n\le N-1$ and $1\le k\le2(N-n-2)$. Note that the upper bound for the coefficient $k$ follows from the conditions $k\le 2\beta$ and $\beta\le m_{n'}$ with $n'>n$. In particular,  the smaller $n$ is, the more spatial derivatives of $U_n$ are required, due to the nested structure of our scheme. For $k=0$, we set $m_n^0 \coloneqq m_n$ to unify the notation, although this does not match the general expression for $m_n^k$ with $k\geq 1$. This distinction is due to the fact that the initial accuracy requirement for $U_n$ imposes a stricter condition than the one arising from the propagation of error. In order to compute these terms, we similarly use
\begin{align}
\nabla^kU_n(t_j)&=-i\sum_{\alpha=0}^{j-1}\sum_{\beta=0}^{m_n^k}\sum_{a\in\CT_n^\beta}c(a)\int_{t_\alpha}^{t_{\alpha+1}}\frac{(s-t_\alpha)^\beta}{\beta!}S(t_j-t_\alpha)\nabla^ka^t(t_\alpha)\drm s\\ 
&\quad-i\sum_{\alpha=0}^{j-1}\sum_{\beta=0}^{m_n^k}\sum_{a\in\CT_n^{m_n^k+1}}c(a)\int_{t_\alpha}^{t_{\alpha+1}}\int_{t_\alpha}^{s_1}\frac{(s_1-s_2)^{m_n^k}}{m_n^k!}S(t_j-s_2)\nabla^ka^t(s_2)\drm s_2\drm s_1 
\end{align}
where $\nabla^ka^t$ is again expressed as a sum of decorated trees, with the weight increased by $k$ ia the Leibniz rule, giving the final sets of decorated trees that appear in the Taylor expansion of $\nabla^kU_n$.

\medskip

\begin{definition}
For $k,\beta\ge0$, the set $\CT_n^{\beta,k}$ consists of trees in $\CT_n^\beta$ where a total weight $k$ is distributed among the edges of the first floor.
\end{definition}

\medskip

Since $a^t$ is a product of lower-order decorated trees for $a\in\CT_n^\beta$, as given by expression~\eqref{eq:arbrecut}, the same holds for elements of $\CT_n^{\beta,k}$, with the only difference being an additional total weight $k$. While this could, in principle, require introducing new spatial derivatives, this is not the case: indeed, the highest-order spatial derivative appearing in $\nabla^kU_n$ is of order $2m_n^k+k$, and we have
\begin{equation}
2m_n^k+k\le N-n-2
\end{equation}
which matches exactly the condition $k\le N-n-2$ from \eqref{ConditionNestedErrors}. As a result, we obtain
\begin{align}
\nabla^kU_n(t_j)&=-i\sum_{\alpha=0}^{j-1}\sum_{\beta=0}^{m_n^k}\sum_{a\in\CT_n^{\beta,k}}c_k(a)S(t_j-t_\alpha) a^t(t_\alpha)\int_{t_\alpha}^{t_{\alpha+1}}\frac{(s-t_\alpha)^\beta}{\beta!}\drm s\\ 
&\quad-i\sum_{\alpha=0}^{j-1}\sum_{\beta=0}^{m_n^k}\sum_{a\in\CT_n^{m_n^k+1,k}}c_k(a)\int_{t_\alpha}^{t_{\alpha+1}}\int_{t_\alpha}^{s_1}\frac{(s_1-s_2)^{m_n^k}}{m_n^k!}S(t_j-s_2) a^t(s_2)\drm s_2\drm s_1,
\end{align}
for $0\le n\le N-1$ and $0\le k\le N-n-2$. This leads to the definition of the \eqref{NTS} scheme $\mathcal{U}_n^{j,k}$, where we approximate
\begin{equation}
\mathcal{U}_n^{j,k}\simeq\nabla^kU_n(t_j)
\end{equation}
for $0\le j\le J$. While a tree $a\in\CT_n^{\beta,k}$ represents a spacetime functions, we now introduce its discretized counterpart $a_j$, defined recursively. This recursive procedure defines the structure of our Nested Taylor Scheme \eqref{NTS}:

\medskip

\begin{definition}
We consider the family $(\mathcal{U}_n^{j,k})_{n,j,k}$ defined recursively as
\begin{equation} \label{NTS} \tag{NTS}
\mathcal{U}_n^{j,k} \coloneqq -i\sum_{\alpha=0}^{j-1}\sum_{\beta=0}^{m_n^k}\sum_{a\in\CT_n^{\beta,k}}c_k(a)\frac{\tau^{\beta+1}}{(\beta+1)!}S_\tau(t_j-t_\alpha)a_\alpha^\cut
\end{equation}
where $a_\alpha^\cut$ is defined as a product of discretized $\mathcal{U}_{n'}^{\alpha,k'}$ where $n<n'$, and with initialization
\begin{equation}
a_j \coloneqq S_{\tau}(t_j)\nabla^k \varphi
\end{equation}
for $0\le j\le J$ and $0\le k\le 2(N-2)$ with $a=\begin{tikzpicture}[baseline]
\node (1) at (0,0) [inner sep=0pt] {};
\node (2) at (0,3/8) [tdot] {};
\node at (1/6,3/16) [] {\tiny$k$};
\draw (1) -- (2);
\end{tikzpicture}$.
\end{definition}

\medskip

Finally, we set the convention
\begin{equation}
\mathcal{U}_n^j \coloneqq \mathcal{U}_n^{0,j}
\end{equation}
for $0\le j\le J$, which will stand as the main quantity of interest to state the following convergence result.

\medskip

\begin{theorem}\label{theorem_convergence_NTS}
For $N\ge1$ (or $1\leq N \leq 3$ if $d=3$) and $\varphi\in\Sigma\cap H^{2N}(\IR^d)$, let $u$ be the solution to equation \eqref{NLS} with initial data $u(0)=\varphi$. Fix $T,\tau>0$ and $J\in\IN$ such that $T=J\tau$ and consider the \eqref{NTS} scheme $(\mathcal{U}_n^j)_{n,j}$ for $0\le n\le N-1$ and $0\le j\le J$ defined previously. Then there exists a constant $C=C(N,d,\|\varphi\|_{\Sigma},\|\varphi\|_{H^N})>0$ independent of the time $T>0$ such that
\begin{equation}
\sup_{0\le j\le J}\|u(t_j)-\sum_{n=0}^{N-1}\eps^n \mathcal{U}_n^j\|_{L^2(\IR^d)}\lesssim \sum_{n=0}^{N}\eps^n\tau^{N-n}.
\end{equation}
In particular, we get
\begin{equation}
\sup_{0\le j\le J}\|u(t_j)-\sum_{n=0}^{N-1}\eps^n \mathcal{U}_n^j\|_{L^2}\lesssim \eps^N
\end{equation}
for $\tau\le\eps$.
\end{theorem}

\begin{remark} \label{remark_NTS_d_3}
The restriction $1\leq N \leq 3$ in dimension $d=3$ (where we are restricted to the cubic case $p=3$) is due to the fact that we we cannot apply rule \textbf{(C)} as it would generate additional terms, such as quintic interactions at the first iteration, that cannot be controlled using the weighted Sobolev inequality associated with the operator $J(t)$ introduced in the next section. We recall that the scheme \eqref{NQS} remains applicable up to order $N \leq 4$ in dimension three.
\end{remark}

\section{Dispersive estimates} \label{section_dispersive_estimates}

In this section, we collect both continuous and discrete dispersive estimates that are essential for proving our main results. Before presenting these estimates, we introduce the operator
\[ J(t) \coloneqq x+ 2 i t \nabla  \]
which play a central role in the scattering theory for nonlinear Schrödinger equations in weighted spaces. This operator satisfies
\begin{equation} \label{eq_prop1_J}
J(t)=S(t) x S(-t)
\end{equation} 
so in particular $J$ commutes with the linear part of \eqref{NLS}, and it can be factorized as
\[ J(t)=2it e^{i\frac{|x|^2}{4t}}  \nabla \left( e^{-i\frac{|x|^2}{4t}}  \cdot \right).  \]
This last property enables us to write a particular weighted Sobolev inequality: for $2 \leq r < \frac{2d}{(d-2)_+}$ for $d\geq 2$, or $2 \leq r \leq \infty$ if $d=1$, there exists $C=C(d,r)>0$ such that 
\begin{equation} \label{eq_weighted_Sobolev}
 \| f \|_{L^r_x} \leq \frac{C}{|t|^{\delta}} \| f \|_{L^2_x}^{1- \delta} \| J(t) f \|_{L^2_x}^{\delta}, \quad \delta = d \left( \frac12 - \frac{1}{r}  \right).
\end{equation}   
One can also remark that if  $F(z)=G(|z|^2)z	$ is $\mathcal{C}^1$, then the operator $J(t)$ acts like a derivative on~$F(\omega)$, which means that
\[ J(t) (F(\omega))= \partial_z F(\omega) J(t) \omega - \partial_{\overline{z}} F(\omega) \overline{J(t)\omega}.  \]
We recall the product rule in Sobolev spaces
\begin{equation} \label{eq_algebra_sobolev}
\| f g \|_{H^{\sigma}_x} \lesssim \|f \|_{H^{\sigma}_x} \| g \|_{H^{\delta}_x}
\end{equation} 
for $\delta >d/2$ and $\sigma \geq 0$. One also has fractional Leibniz rule with $D^{\sigma}$ the Fourier multiplier such that
\begin{equation} 
\widehat{D^{\sigma} \phi }(\xi) = |\xi|^{\sigma}  \widehat{\phi }(\xi)
\end{equation}
for $\xi \in \IR^d$. For $\sigma>0$, we have
\begin{equation} \label{eq_fractional_Leibniz_rule}
  \| D^{\sigma}(f_1 \ldots f_p ) \|_{L^{q_0}_x} \lesssim \sum_{\ell=1}^p \| D^{\sigma} f_{\ell} \|_{L^{q_{n}}_x} \Big( \prod_{\substack{K=1 \\ K \neq \ell}}^p \| f_{K} \|_{L^{q_K}_x} \Big)
  \end{equation} 
for any $p \in \IN^*$ such that $1<q_{n} < \infty$ for $1 \leq n \leq p$ and
\begin{equation} 
\frac{1}{q_0} = \frac{1}{q_1}+ \ldots + \frac{1}{q_p}.
\end{equation}
We finally recall the following useful equivalence of norm which is a direct consequence from interpolation theory in Sobolev spaces, namely
\[ \| f \|_{W^{\sigma,r}_x} \simeq \| D^{\sigma}f \|_{L^r_x}+\| f \|_{L^r_x} \]
for any $\sigma \geq 0$ and $1< r < \infty$.

\subsection{Continous dispersive estimates}

We begin by recalling Strichartz estimates which are crucially used here, standing as a well-known result which trades integrability between time and space with the norms
\begin{equation}
\|f\|_{L_t^qL_x^r}=\Big(\int_{\IR}\Big(\int_{\IR^d}|f(t,x)|^r\drm x\Big)^{\frac{q}{r}}\drm t\Big)^{\frac{1}{q}}
\end{equation}
for $q,r\in[1,\infty)$. In particular, these estimates are the main argument to prove analyticity of the scattering operator, see \cite{Carles2009}. In this context, a pair $(q,r)$ is called \textit{admissible} if
\begin{equation} 
\frac{2}{q}+\frac{d}{r} = \frac{d}{2}
\end{equation}
with $q>2$ in the case $d=2$. The following bounds are respectively called homogeneous and inhomogeneous Strichartz estimates, see \textsc{Keel} and \textsc{Tao} \cite{Kee1998}.

\medskip

\begin{lemma}\label{lemma_KeelTao}
For any $(q,r)$, $(q_1,r_1)$ and $(q_2,r_2)$ admissible pairs, we have constants $C_{d,q}>0$ and $C_{d,q_1,q_2}>0$ such that
\begin{equation}
\|S(t)\varphi\|_{L_t^{q}L_x^{r}}\leq C_{d,q} \|\varphi\|_{L^2}
\end{equation}
and
\begin{equation}
\Big\|\int_{0}^tS(t-s)F(s)\drm s \Big\|_{L_t^{q_1}L_x^{r_1}}\leq C_{d,q_1,q_2} \|F\|_{L_t^{q_2'}L_x^{r_2'}}
\end{equation}
where $q_2'$ and $r_2'$ respectively denote the conjugated exponent of $q_2$ and $r_2$.
\end{lemma}

\medskip

\begin{remark}
Note that in full generality one could state inhomogeneous Strichartz estimates on $L^{q_1}(I; L^{r_1}(\IR^d))$ for any interval $I\subset\IR$, or with integration over $\left[-\infty,t\right]$ instead of $\left[0,t\right]$ and recover the given estimate by applying inhomogeneous Strichartz inequality on $G(s)=F(s)\mathds{1}_{0\leq s \leq t}$. The same remark can be make for the upcoming discrete-in-time norms, see \cite[Remark 2.1]{Ignat2011} or \cite[Remark 4.1]{Ignat2011}.
\end{remark}

\medskip

It can be used to prove that the $U_n$ satisfy the following uniform bounds.

\medskip

\begin{proposition} \label{prop_unif_space_time_bounds_Un}
For $(q,r)$ an admissible pair, $A \in \left\{ \Id, \nabla\right\}$ and any $\sigma\ge0$, we have 
\begin{equation} 
\|AU_n(t)\|_{L^q_t W_x^{\sigma,r}} \lesssim \| A \varphi \|_{H^\sigma_x}^{(p-1)n+1}
\end{equation}
and 
\begin{equation} 
\|J(t)U_n(t)\|_{L^q_t W_x^{\sigma,r}} \lesssim \| x \varphi \|_{H^\sigma_x}^{(p-1)n+1}
\end{equation}
for any $n\ge0$.
\end{proposition}

\medskip

\begin{proof}
The result is proved by induction on $n\ge0$. For $n=0$, if $A\in \left\{ \Id,\nabla \right\}$, as $A$ and $D^{\sigma}$ commute with the linear flow $S(t)$, this simply corresponds to the homogeneous Strichartz estimate
\begin{equation}
\|D^{\sigma}A S(t)\varphi\|_{L^q_tL_x^r}\lesssim \|D^{\sigma}A\varphi\|_{L^2_x} \lesssim \| A\varphi\|_{H^\sigma_x}
\end{equation}
from \cite{Kee1998}. If $A=J(t)$, we analogously write using equation \eqref{eq_prop1_J} that
\begin{equation}
\|D^{\sigma}J(t) S(t)\varphi\|_{L^q_tL_x^r}=\| S(t)D^{\sigma}(x\varphi)\|_{L^q_tL_x^r}\lesssim  \| x\varphi\|_{H^\sigma_x}
\end{equation}
For $n\ge1$, let first take $A\in \left\{ \Id,\nabla,J \right\}$, and let $(q,r)$ be any admissible pair. Let's first note that from property \eqref{eq_prop1_J} we infer
\[ J(t)U_n(t)=-iJ(t) \int_0^t S(t-s) F_n(s) \drm s = -i \int_0^t S(t-s) J(s) F_n(s) \drm s.   \]
We successively apply inhomogeneous Strichartz estimates, fractional Leibniz rule in space and Hölder inequality in time on Duhamel's formula for $U_n$ from equation \eqref{eq_Un_integrated}, so that
\begin{align}
\| D^{\sigma} A U_n \|_{L^{q_1}_t L_x^{r_1}} \lesssim & \sum_{n_1+\ldots+n_p=n-1} \sum_{\ell=1}^p \left\| D^{\sigma} A U_{n_{\ell}} \left( \prod_{\substack{K=1 \\ K \neq \ell}}^p U_{n_K}  \right) \right\|_{L^{q_2'}_t L^{r_2'}_x} \\
&+\sum_{n_1+\ldots+n_p=n-1} \sum_{\substack{\ell,\ell'=1\\ \ell \neq \ell'}}^p \left\|  A U_{n_{\ell}} D^{\sigma} U_{n_{\ell'}} \left( \prod_{\substack{K=1 \\ K \neq \ell,\ell'}}^p U_{n_K}  \right) \right\|_{L^{q_2'}_t L^{r_2'}_x} \\
& \eqqcolon I_1 + I_2
\end{align}
for admissible pairs $(q_1,r_1)$ and $(q_2,r_2)$ yet to be fixed, where we harmlessly omit the complex conjugation in the above formula for clearness purposes. Note that if $\sigma=0$ on $A=\Id$, the second sum $I_2$ vanishes. We introduce the admissible pair
\[ (q,r) = \left( \frac{4(p+1)}{d(p-1)}, p+1 \right)   \]
 and we define $\gamma$ such that
 \[ \gamma = \frac{2(p-1)(p+1)}{4-(d-2)(p-1)} \geq q  \]
as soon as $p \geq 1+ \frac{4}{d}$, so that we get 
\[ \frac{1}{q_2'}=\frac{1}{q} + \frac{p-1}{\gamma} \quad  \text{and} \quad \frac{1}{r_2'}= \frac{p}{r}. \]
To handle the first sum $I_1$, using Hölder's inequality we then write
\[ I_1 \leq \sum_{n_1+ \ldots + n_p=n-1} \sum_{\ell=1}^p  \|D^{\sigma} A U_{n_{\ell}} \|_{L^q_t L^r_x} \prod_{\substack{K=1 \\ K \neq \ell}}^p \| U_{n_K} \|_{L^{\gamma}_t L^r_x}.  \]
From induction hypothesis, we already know that $\| D^{\sigma} A U_{n_1} \|_{L^q_t L^r_x} \lesssim C$, so we need to bound the other terms of the product. This is achieved using the weighted Sobolev inequality \eqref{eq_weighted_Sobolev}, as
\[ \| U_{n_K}(t) \|_{L^r_x} \lesssim \frac{1}{|t|^{\delta}} \| U_{n_K}(t) \|_{L^2_x}^{1-\delta} \|J(t) U_{n_K}(t) \|_{L^2_x}^{\delta} \quad \text{with} \ \delta= \frac{d (p-1)}{2(p+1)} \]
so that $\gamma \delta >1$ as $p \geq 1+\frac{4}{d}$, hence $\| U_{n_K} \|_{L^{\gamma}_t L^r_x} \leq C$ by induction taking the admissible pair $(\infty,2)$ for $U_{n_K}$ and $J U_{n_K}$. We now turn our attention to $I_2$, performing similarly as for $I_1$. We first remark that since 
\[  1+\frac{4}{d} \leq p < 1+\frac{4}{(d-2)_+},\]
 there exists $\rho \geq 2$ such that $(\gamma,\rho)$ is admissible and such that
\[  d \left(\frac{1}{\rho}-\frac{1}{r} \right)=\frac{2}{q}-\frac{2}{\gamma}=\frac{\frac{d(p-1)}{2}-2}{p-1} \eqqcolon \delta \in \left[0,1\right).  \]
 By Hölder inequality we then write that
\[ I_2 \leq \sum_{n_1+ \ldots + n_p=n-1} \sum_{\ell,\ell'=1}^p  \sum_{\substack{\ell,\ell'=1\\ \ell \neq \ell'}}^p \|  A U_{n_{\ell}} \|_{L^q_t L^r_x} \| D^{\sigma} U_{n_{\ell'}} \|_{L^{\gamma}_t L^r_x} \left( \prod_{\substack{K=1 \\ K \neq \ell,\ell'}}^p \| U_{n_K} \|_{L^{\gamma}_t L^r_x}  \right).   \]
The first term $\|  A U_{n_{\ell}} \|_{L^q_t L^r_x}$ is bounded by induction, and so is the second term
\[ \| D^{\sigma} U_{n_{\ell'}} \|_{L^{\gamma}_t L^r_x} \lesssim \| D^{\sigma} U_{n_{\ell'}} \|_{L^{\gamma}_t W^{\delta,\rho}_x} \lesssim \| U_{n_{\ell'}} \|_{L^{\gamma}_t W^{\delta+\sigma,\rho}_x} \]
using the Sobolev embedding $W^{\delta,\rho}_x \hookrightarrow L^r_x$ with $\delta$ defined as above, as $(q,r)$ and $(\gamma,\rho)$ are both admissible pairs. The other terms $\| U_{n_K} \|_{L^{\gamma}_t L^r_x}$ are handled the same way as for $I_1$ using the weighted Sobolev inequality \eqref{eq_weighted_Sobolev}, which ends the proof.
\end{proof}

\subsection{Truncation of high frequencies}

We first gather useful bounds related to the low frequency projection $\Pi_{\tau}$ which are a direct consequence of Bernstein's lemma, see for instance Lemma $2.1$ in \cite{BCD}.

\medskip

\begin{lemma}\label{lemma_projection_Fourier}
Let $\sigma,\delta\ge0$. For $1 \leq r < \infty$ and $\phi : \IR^d \rightarrow \IC$, we have
\begin{equation}
\| \Pi_{\tau} \phi - \phi \|_{W^{\sigma,r}_x} \leq C \tau^{\frac{\delta}{2}} \| \phi \|_{W^{\sigma+\delta,r}_x}
\end{equation}
and
\begin{equation}
\| \Pi_{\tau} \phi \|_{W^{\sigma,r}_x} \leq C \| \phi \|_{W^{\sigma,r}_x} .
\end{equation}
\end{lemma}

We now state a slight generalization of the commutator estimate between $J$ and $\Pi_{\tau}$, originally given in \cite[Lemma 3.3]{Carles2024}.

\begin{lemma}\label{lemma_commutator_J_projector}
Let $1<r< \infty$ and $\sigma,\delta \geq 0$, then 
\[ \| J(t) \Pi_{\tau}\phi- \Pi_{\tau} J(t)\phi \|_{W^{\sigma,r}_x} \lesssim  \tau^{\frac{1+\delta}{2}} \| \phi \|_{W^{\sigma+\delta,r}_x} \]
for all $\phi \in \Sigma \cap W^{\sigma+\delta,r}(\IR^d)$ and all $t\in \IR$.
\end{lemma}
\begin{proof}
Following the proof of \cite[Lemma 3.3]{Carles2024}, we directly compute that for all $\xi \in \IR^d$,
\[  \widehat{J(t) \Pi_{\tau}\phi}(\xi)-\widehat{ \Pi_{\tau} J(t)\phi}(\xi) = i\sqrt{\tau} \nabla \chi (\sqrt{\tau} \xi) \widehat{\phi}(\xi) \]
and as $\nabla \chi$ is a smooth cut-off function located on a ring of size $\sqrt{\tau}$, we infer
\[ | \widehat{J(t) \Pi_{\tau}\phi}(\xi)-\widehat{ \Pi_{\tau} J(t)\phi}(\xi) | \lesssim \tau^{\frac{1+\delta}{2}} |\nabla \chi (\sqrt{\tau} \xi) \langle \xi \rangle^{\delta}\widehat{\phi}(\xi)|. \]
The result follows from Fourier multiplier theory \cite{BCD}.
\end{proof}

\medskip

We are now going to prove continuous dispersive estimates associated to the projected linear flow $S_{\tau}$, based on the following lemma from \cite[Theorem 2.1 (i)]{Ignat2011}.

\begin{lemma} \label{lemma_discrete_flow_continuous_Strichartz}
For any $(q,r)$, $(q_1,r_1)$ and $(q_2,r_2)$ admissible pairs and for $\tau>0$, we have
\begin{equation}
\| S_{\tau}(t) \varphi \|_{L^{q}_t L^r_x} \lesssim \|\varphi \|_{L^2_x}
\end{equation}
and
\begin{equation}
\Big\| \int_0^t S_{\tau}(t-s) F(s) \drm s\Big\|_{L^{q_1}_t L^{r_1}_x} \lesssim \| F \|_{L^{q_2'}_t L^{r_2'}_x}
\end{equation}
for all $\varphi \in L^2(\IR^d)$ and $F \in L^{q_2'}(\IR; L^{r_2'}(\IR^d))$. 
\end{lemma}

\begin{corollary}\label{corollary_discrete_flow_continuous_Strichartz}
For any $(q,r)$, $(q_1,r_1)$ and $(q_2,r_2)$ admissible pairs, for $\tau>0$ and for $A \in \left\{ \Id,\nabla,J \right\}$, we have
\begin{equation}
\| A S_{\tau}(t) \varphi \|_{L^{q}_t L^r_x} \lesssim \| \varphi \|_{\Sigma}
\end{equation}
and
\begin{equation}
\Big\| A \int_0^t S_{\tau}(t-s) F(s) \drm s\Big\|_{L^{q}_t L^{r}_x} \lesssim \| F \|_{L^{q_2'}_t L^{r_2'}_x} + \| A F \|_{L^{q_2'}_t L^{r_2'}_x}
\end{equation}
for all $\varphi \in \Sigma$ and $F,AF \in L^{q_2'}(\IR; L^{r_2'}(\IR^d))$. 
\end{corollary}
\begin{proof}
The case $A=\Id$ corresponds to Lemma \ref{lemma_discrete_flow_continuous_Strichartz}, and the case $A=\nabla$ is straightforward as $\nabla$ and $S_{\tau}$ commute. We are then left with the case $A=J$, which does not commute with $S_{\tau}$. We simply write the commutator identity
\[ J(t)S_{\tau}(t-s)\varphi= S(t-s) J(s) \Pi_{\tau} \varphi= S(t-s) \left[J(s),\Pi_{\tau} \right]\varphi + S_{\tau}(t-s)J(s)\varphi \]
with the standard notation $\left[X,Y\right]=XY-YX$, hence we get the homogeneous Strichartz estimate (taking $s=0$)
\begin{align*}
\| J(t) S_{\tau}(t)\varphi \|_{L^q_t L^r_x} & \lesssim   + \| S_{\tau}(t) \left[x,\Pi_{\tau} \right] \varphi \|_{L^q_t L^r_x} + \| S_{\tau}(t) x \varphi \|_{L^q_t L^r_x} \\
& \lesssim  \| \left[x,\Pi_{\tau} \right] \varphi \|_{ L^2_x} + \|  x \varphi \|_{L^2_x} \\
& \lesssim \sqrt{\tau} \| \varphi \|_{L^2_x} + \|  x \varphi \|_{L^2_x}
\end{align*}
where we have used a combination of Lemma \ref{lemma_commutator_J_projector} with $\sigma=\delta=0$ alongside Lemma \ref{lemma_KeelTao} and Lemma~\ref{lemma_discrete_flow_continuous_Strichartz}, which gives the result as $\tau \leq 1$. The inhomogeneous Strichartz estimate is established similarly.
\end{proof}

In order to prove the convergence of our numerical schemes, we will compare them with the projected iterates $(U_n^{\tau})_{0\le n\le N-1}$ defined recursively for all $t \in \IR$ as $U_0^{\tau}(t)=S_\tau(t)\varphi$ and
\begin{equation}
U_n^{\tau}(t) = -i\int_0^tS_\tau(t-s)F_n^{\tau}(s)\drm s
\end{equation}
for $n\ge1$, where 
\[F_n^{\tau}(t) = \sum_{n_1+\ldots+n_p=n-1} U_{n_1}^{\tau}(t) \overline{U_{n_2}^{\tau}(t)} \ldots U_{n_p}^{\tau}(t).   \]
We first prove some dispersive bounds for this truncated family.

\begin{proposition} \label{prop_unif_space_time_bounds_Un^tau}
For $(q,r)$ an admissible pair, $A \in \left\{ \Id, \nabla\right\}$ and any $\sigma\ge0$, we have 
\begin{equation} 
\|A U_n^{\tau}(t)\|_{L^q_t W_x^{\sigma,r}} \lesssim \| A \varphi \|_{H^\sigma_x}^{(p-1)n+1}
\end{equation}
and 
\begin{equation} 
\|J(t)U_n^{\tau}(t)\|_{L^q_t W_x^{\sigma,r}} \lesssim \| x \varphi \|_{H^\sigma_x}^{(p-1)n+1}
\end{equation}
for any $n\ge0$.
\end{proposition}
\begin{proof}
Once again, we perform by induction. The initialization simply corresponds to Corollary~\ref{corollary_discrete_flow_continuous_Strichartz}. The proof of the induction argument follows closely the one of Proposition~\ref{prop_unif_space_time_bounds_Un}, noting in particular that the differentiation operator $D^{\sigma}$ commutes with the projected linear flow $S_{\tau}$. The only new element is the appearance of the commutators between the operator $J$ and the projector $\Pi_{\tau}$, leading to the identity
 \begin{align*}
  J(t)U_n^{\tau}(t) & =-i\int_0^t S(t-s)J(s)\Pi_{\tau} F_n^{\tau}(s) \drm s  \\
  & = -i \int_0^t S_{\tau}(t-s) J(s) F_n^{\tau}(s) \drm s -i \int_0^t S(t-s) \left[ J(s),\Pi_{\tau} \right] F_n{\tau}(s) \drm s   
  \end{align*}
as in the proof of Corollary \ref{corollary_discrete_flow_continuous_Strichartz}. While the first term is treated inductively as in Proposition~\ref{prop_unif_space_time_bounds_Un}, using the Strichartz estimates from Corollary~\ref{corollary_discrete_flow_continuous_Strichartz}, the second term also use the commutator bound in Lemma~\ref{lemma_commutator_J_projector}, followed by an application of Strichartz estimates from Lemma~\ref{lemma_KeelTao} combined with the induction hypothesis.
\end{proof}

The following proposition ensures that the truncated sequence $(U_n^{\tau})_{0\le n\le N-1}$ provides an accurate enough approximation of the continuous sequence $(U_n)_{0\le n\le N-1}$. Although our final convergence results are established in the $L_t^\infty L_x^2$ space, the inductive proofs of our main theorems will require to work in semi-discrete norms $\ell_\tau^q L_x^r$, for which the linear flow $S(t)$ no longer satisfies dispersive estimates, which motivates the use of the truncated iterates $U_n^{\tau}$.

\medskip 

\begin{proposition}\label{prop:Unstarbound}
For any admissible pair $(q,r)$ and for $A \in \left\{ \Id,\nabla,J \right\}$, we have
\begin{equation}
\|A(U_n(t)-U_n^{\tau}(t))\|_{L_t^qL_x^r}\lesssim\tau^{N}\|\varphi\|_{\Sigma \cap H^{2N}_x }
\end{equation}
for any $n\ge0$.
\end{proposition}

\medskip 

\begin{proof}
We prove the result by induction on $n\ge0$. We first have for $A \in \left\{ \Id,\nabla \right\}$,
\begin{equation}
\|A(U_0(t)-U_0^{\tau}(t))\|_{L_t^qL_x^r}=\|S(t)(\Id-\Pi_\tau)A\varphi\|_{L_t^qL_x^r}\lesssim\|(\text{Id}-\Pi_\tau)A\varphi\|_{L^2_x}\lesssim \tau^N\|A\varphi\|_{H^{2N}_x}
\end{equation}
using successively Lemmas \ref{lemma_KeelTao} and \ref{lemma_projection_Fourier}. If $A=J(t)$, we rather write
\begin{align*}
\|A(U_0(t)-U_0^{\tau}(t))\|_{L_t^qL_x^r} & =\|S(t)x(\Id-\Pi_\tau)\varphi\|_{L_t^qL_x^r}\\
& \lesssim\|x(\Id-\Pi_\tau)\varphi\|_{L^2_x} \\
& \lesssim \| (\Id - \Pi_{\tau}) x \varphi \|_{L^2_x} + \| \left[ x,\Pi_\tau \right] \varphi \|_{L^2_x} \\
& \lesssim \| x \varphi \|_{H^{2N}_x} + \| \left[ x,\Pi_\tau \right] \varphi \|_{L^2_x} \\
& \lesssim \tau^N( \|x \varphi\|_{H^{2N}_x} + \sqrt{\tau}\|\varphi\|_{H^{2N}_x}).
\end{align*}
where we have used Lemmas \ref{lemma_KeelTao}, \ref{lemma_projection_Fourier} and \ref{lemma_commutator_J_projector}. For $n\ge1$, we focus on the case $A=J$, as the other cases $A \in \left\{\Id,\nabla \right\}$ follow similarly and are in fact simpler. We write thanks to commutators as in the proof of Proposition \ref{prop_unif_space_time_bounds_Un^tau} that
\begin{align}
\|J(t)(U_n(t)-U_n^{\tau}(t))\|_{L_t^q L_x^r}&=\Big\|J(t) \int_0^t S(t-s)F_n(s)\drm s- J(t) \int_0^tS_\tau(t-s)F_n^{\tau}(s)\drm s\Big\|_{L_t^qL_x^r}\\ 
& \leq \Big\| \int_0^t S(t-s)(\Id-\Pi_{\tau}) J(s)F_n(s) \drm s \Big\|_{L_t^qL_x^r} \\
& \quad + \Big\| \int_0^t S_{\tau}(t-s) \left[ J(s),\Pi_{\tau} \right] F_n(s) \drm s  \Big\|_{L_t^qL_x^r} \\
& \quad +  \Big\| \int_0^t S_{\tau}(t-s) \left[ J(s),\Pi_{\tau} \right] F_n^{\tau}(s) \drm s  \Big\|_{L_t^qL_x^r} \\
& \quad + \Big\| \int_0^t S_{\tau}(t-s) J(s) (F_n(s)-F_n^{\tau}(s)) \drm s  \Big\|_{L_t^qL_x^r} \\
& \eqqcolon \mathcal{I}_1 + \mathcal{I}_2 + \mathcal{I}_3 + \mathcal{I}_4.
\end{align}
We estimate each term separately. First for $\mathcal{I}_1$, we successively use inhomogeneous Strichartz estimates from Lemma \ref{lemma_KeelTao} with $(q_2,r_2)$ and admissible pair, Bernstein inequality from Lemma \ref{lemma_projection_Fourier}, so that
\[ \mathcal{I}_1 \lesssim  \| (\Id-\Pi_{\tau}) J(t) F_n \|_{L^{q_2'}_t L^{r_2'}_x} \lesssim \tau^N \|J(t) F_n \|_{L^{q_2'}_t W^{2N,r_2'}_x} \]
and we conclude using fractional Leibniz rule alongside Proposition \ref{prop_unif_space_time_bounds_Un} in the same exact way as in the proof of Proposition \ref{prop_unif_space_time_bounds_Un}. The terms $\mathcal{I}_2$ and $\mathcal{I}_3$ are bounded similarly, using first inhomogeneous Strichartz estimates from \ref{lemma_discrete_flow_continuous_Strichartz}, then the commutator estimate between $J$ and $\Pi_{\tau}$ Lemma \ref{lemma_commutator_J_projector}, so for instance
\[  \mathcal{I}_2 \lesssim \| \left[ J,\Pi_{\tau} \right] F_n \|_{L^{q_2'}_t L^{r_2'}_x} \lesssim \tau^N \| F_n \|_{L^{q_2'}_t W^{2N,r_2'}_x} \]
and we again conclude by fractional Leibniz rule and Proposition \ref{prop_unif_space_time_bounds_Un} (for $\mathcal{I}_2$) or Proposition \ref{prop_unif_space_time_bounds_Un^tau} (for $\mathcal{I}_3$). Finally we treat $\mathcal{I}_4$, namely
\begin{multline*}
\mathcal{I}_4  \lesssim  \| J (U_n-U_n^{\tau}) \|_{L^{q_2'}_t L_x^{r_2'}} \\
  \lesssim \sum_{n_1+\ldots+n_p=n-1} \sum_{\ell=1}^p \left( \prod_{K=1}^{\ell-1} \| U_{n_K}\|_{L^{\gamma}_t L^r_x} \right) \|  J (U_{n_{\ell}}-U_{n_{\ell}}^{\tau}) \|_{L^{q}_t L^{r}_x} \left( \prod_{K'=\ell+1}^p \| U_{n_{K'}}^{\tau}\|_{L^{\gamma}_t L^r_x} \right)  \\
+\sum_{n_1+\ldots+n_p=n-1} \sum_{\substack{\ell,\ell'=1\\ \ell'<\ell  }}^p \| J U_{n_{\ell'}} \|_{L^{\gamma}_t L^r_x} \left( \prod_{\substack{K=1\\K\neq \ell'}}^{\ell-1} \| U_{n_K}\|_{L^{\gamma}_t L^r_x} \right) \|U_{n_{\ell}}-U_{n_{\ell}}^{\tau} \|_{L^{q}_t L^{r}_x} \left( \prod_{K'=\ell+1}^p \| U_{n_{K'}}^{\tau}\|_{L^{\gamma}_t L^r_x} \right) \\
+\sum_{n_1+\ldots+n_p=n-1} \sum_{\substack{\ell,\ell'=1\\ \ell'>\ell  }}^p \| J U_{n_{\ell'}}^{\tau} \|_{L^{\gamma}_t L^r_x} \left( \prod_{K=1}^{\ell-1} \| U_{n_K}\|_{L^{\gamma}_t L^r_x} \right) \|U_{n_{\ell}}-U_{n_{\ell}}^{\tau} \|_{L^{q}_t L^{r}_x} \left( \prod_{\substack{K'=\ell+1\\K'\neq \ell'}}^p \| U_{n_{K'}}^{\tau}\|_{L^{\gamma}_t L^r_x} \right) 
\end{multline*}
where the admissible pairs $(q,r)$ and $(\gamma,\rho)$ are taken as in the proof of Proposition \ref{prop_unif_space_time_bounds_Un}. We then apply the induction hypothesis, the weighted Sobolev inequality~\eqref{eq_weighted_Sobolev}, and the uniform bounds from Lemmas~\ref{prop_unif_space_time_bounds_Un} and~\ref{prop_unif_space_time_bounds_Un^tau}. The only new contributions are the terms $\| J U_{n_{\ell'}} \|_{L^{\gamma}_t L^r_x}$ and $\| J U_{n_{\ell'}}^{\tau} \|_{L^{\gamma}_t L^r_x}$, which are estimated once again via the Sobolev embedding  $W^{\delta,\rho}_x \hookrightarrow L^r_x$ with $\delta$ defined as in the proof of Proposition \ref{prop_unif_space_time_bounds_Un}. We conclude by using the uniform estimates from respectively Proposition~\ref{prop_unif_space_time_bounds_Un} and Proposition \ref{prop_unif_space_time_bounds_Un^tau}, which ends the proof.
\end{proof}

\subsection{Discrete dispersive estimates}

We now introduce discrete-in-time Lebesgue spaces
\begin{equation}
\|f\|_{\ell_{\tau}^q L_x^r}=\Big(\tau \sum_{j \in \IZ} \big(\int_{\IR^d}|f(j \tau,x)|^r\drm x \big)^{\frac{q}{r}} \Big)^{\frac{1}{q}}
\end{equation}
for a given time step $\tau >0$. The following discrete Strichartz estimates, serve as discrete analogs of the continuous results from Lemma \ref{lemma_KeelTao}, and have been recently stated in \cite[Corollary 3.4]{Carles2024}.

\medskip

\begin{lemma}\label{lemma_discrete_Strichartz}
Let $(q,r)$, $(q_1,r_1$ and $(q_2,r_2)$ be admissible pairs, let $\tau>0$ and $A \in \left\{ \Id,\nabla \right\}$. Then we have for all $0 \leq j \leq J$ that
\begin{equation}
\| A S_{\tau}(t_j) \varphi \|_{\ell^{q}_{\tau} L^{r}_x} \lesssim \| \varphi \|_{L^2_x}
\end{equation}
and
\begin{equation}
\Big\| \tau A \sum_{k=0}^{j-1} S_{\tau}\big(t_j-t_k\big) F(t_k)\Big\|_{\ell^{q_1}_{\tau} L^{r_1}_x} \lesssim  \| A F \|_{\ell^{q_2'}_{\tau} L^{r_2'}_x}
\end{equation}
as well as
\begin{equation}
\| J(t_j) S_{\tau}(t_j) \varphi \|_{\ell^{q}_{\tau} L^{r}_x} \lesssim \| x\varphi \|_{L^2_x}
\end{equation}
and
\begin{equation}
\Big\| \tau J(t_j) \sum_{k=0}^{j-1} S_{\tau}\big(t_j-t_k\big) F(t_k)\Big\|_{\ell^{q_1}_{\tau} L^{r_1}_x} \lesssim  \sqrt{\tau} \| F \|_{\ell^{q_2'}_{\tau} L^{r_2'}_x} + \| J(t_j) F \|_{\ell^{q_2'}_{\tau} L^{r_2'}_x}
\end{equation}
for all $\varphi \in \Sigma$ and $F \in \ell^{q_2'}(\tau \IZ; L^{r_2'}(\IR^d))$. 
\end{lemma}

\begin{remark}
In fact, the discrete Strichartz estimates above remain valid if the projected linear flow $S_{\tau}(t)$ is replaced by $S_{\kappa \tau}(t)$ for any $\kappa>0$. This flexibility will be useful in later proofs, particularly when working with the additional projector $\Pi_{\tau/4}$ or in establishing convergence of the~\eqref{NQS} scheme.
\end{remark}

\medskip

We also give this version of Strichartz estimates which comes from \cite[Corollary 3.5]{Carles2024}:

\medskip

\begin{lemma}\label{lemma_discrete_continuous_Strichartz}
Let $(q_1,r_1)$ and $(q_2,r_2)$ be admissible pairs, let $\tau>0$ and $A \in \left\{ \Id,\nabla \right\}$. Then we have for all $j \in \IN$ that
\begin{equation}
\Big\| A \int_0^{j \tau} S_{\tau}(j \tau-s) F(s) \drm s \Big\|_{\ell^{q_1}_{\tau} L^{r_1}_x} \lesssim  \| AF \|_{L^{q_2'}_{t} L^{r_2'}_x}
\end{equation}
and 
\begin{equation}
\Big\| J(t_j) \int_0^{j \tau} S_{\tau}(t_j-s) F(s) \drm s \Big\|_{\ell^{q_1}_{\tau} L^{r_1}_x} \lesssim  \sqrt{\tau} \| F \|_{L^{q_2'}_{t} L^{r_2'}_x} + \| J(t_j) F \|_{L^{q_2'}_{t} L^{r_2'}_x}
\end{equation}
for all $F \in L^{q_2'}(\IR ; L^{r_2'}(\IR^d))$. 
\end{lemma}

We finally remark that truncated iterates $(U_n)_n$ also satisfies dispersive estimates in discrete norms.

\begin{proposition}\label{prop_discrete_space_continuous_Strichartz}
For $(q,r)$ an admissible pair and any $\sigma\ge0$, we have 
\begin{equation} 
\|A U_n^{\tau}(t_j)\|_{\ell^q_{\tau} W_x^{\sigma,r}} \lesssim \| A \varphi \|_{H^\sigma_x}^{(p-1)n+1}
\end{equation}
for $A \in \left\{ \Id, \nabla\right\}$ and 
\begin{equation} 
\|J(t_j)U_n^{\tau}(t_j)\|_{\ell^q_{\tau} W_x^{\sigma,r}} \lesssim \| x \varphi \|_{H^\sigma_x}^{(p-1)n+1}
\end{equation}
for any $n\ge0$.
\end{proposition}

\medskip

\begin{proof}
As before, the proof (by induction) follows the structure of Propositions~\ref{prop_unif_space_time_bounds_Un} and~\ref{prop_unif_space_time_bounds_Un^tau}, this time using the discrete Strichartz estimates from Lemmas~\ref{lemma_discrete_continuous_Strichartz} in place of the continuous ones. We provide the proof only for $A=J(t)$, as the cases $A\in\left\{ \Id,\nabla \right\}$  follow similar arguments and are in fact simpler. The main challenge here is that the continuous linear flow $S(t)$ does not satisfy dispersive estimates in discrete time norms $\ell^q_{\tau} L^r_x$, so one cannot make the continuous linear flow $S(t)$ appear in the commutator expressions involving the operator $J(t)$ as in the proof of Proposition~\ref{prop_unif_space_time_bounds_Un^tau}. To overcome this, we adopt the strategy from \cite[Corollary 3.4]{Carles2024}, based on projector composition. Specifically, we write
\[  \Pi_{\tau}=\Pi_{\tau}\Pi_{\tau/4}, \]
which allows us to derive, in the same way as \cite[Corollary 3.4]{Carles2024},
\[ J(t)S_{\tau}(t-s)\varphi= \left[J(t),\Pi_{\tau} \right]S_{\tau/4}(t-s)\varphi +S_{\tau}(t-s) \left[ J(s),\Pi_{\tau/4} \right]\varphi +S_{\tau}(t-s)J(s)\varphi \]
where only the discrete flow  $S_{\tau}$ appears, thus remaining compatible with the discrete Strichartz framework. We then directly get the initialization $n=0$ from the homogeneous Strichartz estimate
\begin{align*}
\| J(t) S_{\tau}(t)\varphi \|_{\ell^q_{\tau} L^r_x} & \lesssim  \| \left[J(t),\Pi_{\tau} \right]S_{\tau/4}(t-s)\varphi \|_{\ell^q_{\tau} L^r_x} + \| S_{\tau}(t) \left[x,\Pi_{\tau/4} \right] \varphi \|_{\ell^q_{\tau} L^r_x} + \| S_{\tau}(t) x \varphi \|_{\ell^q_{\tau} L^r_x} \\
& \lesssim \sqrt{\tau} \| S_{\tau/4}(t-s)\varphi \|_{\ell^q_{\tau} L^r_x} + \| \left[x,\Pi_{\tau/4} \right] \varphi \|_{\ell^q_{\tau} L^r_x} + \|  x \varphi \|_{L^2_x} \\
& \lesssim \sqrt{\tau} \| \varphi \|_{L^2_x} + \|  x \varphi \|_{L^2_x}
\end{align*}
where we have used combinations of Lemma \ref{lemma_commutator_J_projector} with $\sigma=\delta=0$ and Lemma \ref{lemma_discrete_continuous_Strichartz}, which gives the result as $\tau \leq 1$. For $n \geq 1$, we simply write in view of the above identity that
\begin{align*}
J(t_j) U_n^{\tau}(t_j) & = -i J(t_j) \int_0^{t_j} S_{\tau}(t_j-s) F_n^{\tau}(s) \drm s \\
& = - i \int_0^{t_j} \left[J(t_j),\Pi_{\tau}  \right] S_{\tau/4}(t_j -s ) F_n^{\tau}(s) \drm s -i \int_0^{t_j} S_{\tau}(t_j -s ) \left[J(s),\Pi_{\tau/4}  \right]  F_n^{\tau}(s) \drm s \\
& \quad -i \int_0^{t_j} S_{\tau}(t_j -s ) J(s)  F_n^{\tau}(s) \drm s
\end{align*}
so using combinations of Lemma \ref{lemma_commutator_J_projector} with $\delta=0$ and Lemma \ref{prop_discrete_space_continuous_Strichartz}, we infer
\begin{align*}
\|J(t_j)U_n^{\tau}(t_j) \|_{\ell^{q}_{\tau} W^{\sigma,r}_x} & \lesssim \sqrt{\tau} \| \int_0^{t_j} S_{\tau/4}(t_j-s) F_n^{\tau}(s) \drm s \|_{\ell^{q}_{\tau} W^{\sigma,r}_x} +\| \left[ J(s),\Pi_{\tau/4} \right] F_n^{\tau}(s) \drm s \|_{L^{q_2'}_t W^{\sigma,r_2'}_x} \\
& \quad + \| J(t)F_n^{\tau} \|_{L^{q_2'}_t W^{\sigma,r_2'}_x} \\
& \lesssim \| F_n^{\tau} \|_{L^{q_2'}_t W^{\sigma,r_2'}_x}+\| J(t)F_n^{\tau} \|_{L^{q_2'}_t W^{\sigma,r_2'}_x}
\end{align*}
We can then conclude by induction the same way as in the proof of Proposition \ref{prop_unif_space_time_bounds_Un} using the results from Proposition \ref{prop_unif_space_time_bounds_Un^tau}, which ends the proof.
\end{proof}

\subsection{Dispersive estimates for NQS}

We now present discrete dispersive estimates specific to \eqref{NQS}. 

\medskip

\begin{proposition} \label{prop_unif_discrete_space_time_bounds_NQS}
Let $(\mathfrak{U}_n^{j})_{n,j}$ be defined by \eqref{NQS} for $0\leq n \leq 3$. For $(q,r)$ an admissible pair, $A\in \left\{\Id,\nabla \right\}$ and $\sigma \geq 0$, we have
\begin{equation}
\|A \mathfrak{U}_n^{j}\|_{\ell^q_{\tau} W^{\sigma,r}_x} \lesssim \| A\varphi \|_{L^2}^{(p-1)n+1}
\end{equation}
as well as
\begin{equation}
\|J(t_j) \mathfrak{U}_n^{j}\|_{\ell^q_{\tau} W^{\sigma,r}_x} \lesssim \| \varphi \|_{\Sigma}^{(p-1)n+1}.
\end{equation}
\end{proposition}

\medskip

\begin{proof}
For $n=0$, we remark that by definition
\[ \| J(t_j) \mathfrak{U}_0^j \|_{\ell^q_{\tau} W^{\sigma,r}_x} = \| J(t_j) U_0^{\tau}(t_j) \|_{\ell^q_{\tau} W^{\sigma,r}_x}   \]
so the result is directly given by Proposition \ref{prop_discrete_space_continuous_Strichartz}. For $n=1$, considering commutators as before and in view of definition of \eqref{NQS}, we write that
\begin{align*}
J(t_j) \mathfrak{U}_1^j  = &-i\tau \sum_{\alpha=0}^{j-1} \sum_{\beta=0}^2 \left[ J(t_j),\Pi_{\tau} \right] S_{\tau/4}(t_j-t_{\alpha,\beta}^{(2)}) \left( \omega_{\beta}^{(2)} \mathfrak{F_1}^{\alpha+\frac{\beta}{2}}  \right) \\
& -i\tau \sum_{\alpha=0}^{j-1} \sum_{\beta=0}^2 S_{\tau}(t_j-t_{\alpha,\beta}^{(2)}) \left[ J(t_{\alpha,\beta}^{(2)}),\Pi_{\tau/4} \right]  \left( \omega_{\beta}^{(2)} \mathfrak{F_1}^{\alpha+\frac{\beta}{2}}  \right) \\
& -i\tau \sum_{\alpha=0}^{j-1} \sum_{\beta=0}^2 S_{\tau}(t_j-t_{\alpha,\beta}^{(2)})  \left( \omega_{\beta}^{(2)} \mathfrak{F_1}^{\alpha+\frac{\beta}{2}}  \right).
\end{align*}
As $S_{\tau}(t_j-t_{\alpha,\beta}^{(2)})$ and $S_{\tau/4}(t_j-t_{\alpha,\beta}^{(2)})$ are applied on regular grids with step size $\tau/2$, one can apply discrete Strichartz estimates from Lemma \ref{lemma_discrete_Strichartz}, alongside commutator error from Lemma \ref{lemma_commutator_J_projector}, which gives
\[ \| J(t_j) \mathfrak{U}_1^j \|_{\ell^q_{\tau} W^{\sigma,r}_x} \leq \| J(t_{\alpha}) \mathfrak{U}_1^{\alpha} \|_{\ell^q_{\tau/2} W^{\sigma,r}_x} \lesssim \| \mathfrak{R}^{\alpha} \|_{\ell^{q_2'}_{\tau/2} W^{\sigma,r_2'}_x} + \| J(t_{\alpha} \mathfrak{R}^{\alpha} \|_{\ell^{q_2'}_{\tau/2} W^{\sigma,r_2'}_x} \]
where
\[ 
\left\{
\begin{aligned}
&\mathfrak{R}^{\alpha}\coloneqq \omega_{0}^{(2)} \mathfrak{F}_1^{\alpha} + \omega_{0}^{(2)}\mathfrak{F}_1^{\alpha +1} = \frac{1}{3} \left( |\mathfrak{U}_0^{\alpha}|^{p-1} \mathfrak{U}_0^{\alpha} + |\mathfrak{U}_0^{\alpha+1}|^{p-1} \mathfrak{U}_0^{\alpha+1} \right) \\
&\mathfrak{R}^{\alpha+\frac12}\coloneqq \omega_{1}^{(2)} \mathfrak{F}_1^{\alpha+\frac12} = \frac{2}{3} |\mathfrak{U}_0^{\alpha+\frac12}|^{p-1} \mathfrak{U}_0^{\alpha+\frac12}
\end{aligned} \right.   \]
for $\alpha \in \tau \IZ$. It is then direct to write that
\[ \| \mathfrak{R}^{\alpha} \|_{\ell^{q_2'}_{\tau/2} W^{\sigma,r_2'}_x} \lesssim \| \mathfrak{F}_1^{\alpha} \|_{\ell^{q_2'}_{\tau} W^{\sigma,r_2'}_x} + \| \mathfrak{F}_1^{\alpha+\frac12} \|_{\ell^{q_2'}_{\tau} W^{\sigma,r_2'}_x}.  \]
The first term in the above estimate is handled classically by a combination of Hölder inequality and weighted Sobolev inequality as in the proof of Proposition \ref{prop_unif_space_time_bounds_Un}. We now remark that as 
\[ \| S_{\tau}(t_j) S_{\tau}(\tau/2) \varphi \|_{\ell^q_{\tau} L^r_x} \lesssim \| S_{\tau}(\tau/2) \varphi \|_{L^2_x} \leq  \| \varphi \|_{L^2_x} \]
by discrete Strichartz estimates Lemma \ref{lemma_discrete_Strichartz} and continuity of $\Pi_{\tau}$ \ref{lemma_projection_Fourier}, the second term satisfies the same discrete dispersive estimates as the first term, and we can conclude similarly. The proofs for~$(\mathfrak{U}_2^j)_j$ and $(\mathfrak{U}_3^j)_j$ follow the same lines, and are in fact simpler, as they do not require the use of finer temporal grids.
\end{proof}

\subsection{Dispersive estimates for NTS}

We also prove discrete dispersive estimates for \eqref{NTS}.

\medskip

\begin{proposition} \label{prop_unif_discrete_space_time_bounds_NTS}
Let $(\mathcal{U}_n^{j,k})_{n,j,k}$ be defined by \eqref{NTS} for $0\leq n \leq N-1$. For $(q,r)$ an admissible pair, $A\in \left\{\Id,\nabla \right\}$ and $\sigma \geq 0$, we have
\begin{equation}
\|A \mathcal{U}_n^{j}\|_{\ell^q_{\tau} W^{\sigma,r}_x} \lesssim \| A\varphi \|_{H^k_x}^{(p-1)n+1}
\end{equation}
as well as
\begin{equation}
\|J(t_j) \mathcal{U}_n^{j}\|_{\ell^q_{\tau} W^{\sigma,r}_x} \lesssim \| \varphi \|_{\Sigma\cap H^k_x}^{(p-1)n+1}.
\end{equation}
\end{proposition}

\medskip

\begin{proof}
The proof is once again made by induction on $n$. For $n=0$, we directly get that
\[  \| J(t_j) \mathcal{U}_0^{j,k} \|_{\ell^q_{\tau} W^{\sigma,r}_x} = \| J(t_j) S_{\tau}(t_j) \nabla^k \varphi \|_{\ell^q_{\tau} W^{\sigma,r}_x} \lesssim \| \nabla^k \varphi \|_{H^{\sigma}_x} + \| x \nabla^k \varphi \|_{H^{\sigma}_x}  \]
using commmutators estimates from Lemma \ref{lemma_commutator_J_projector} and discrete Strichartz estimates from Lemma \ref{lemma_discrete_Strichartz}. For $n \geq 1$, we first write that
\[ \| J(t_j) \mathcal{U}_n^{j,k} \|_{\ell^q_{\tau} W^{\sigma,r}_x} \leq \sum_{\beta=0}^{m_n^k} \sum_{a \in \mathcal{T}_n^{\beta,k}} \tau^{\beta} \| \tau J(t_j) \sum_{\alpha=0}^{j-1} S_{\tau}(t_j-t_{\alpha} a_{\alpha}^\cut \|_{\ell^q_{\tau} W^{\sigma,r}_x}. \]
Since $a_k^\cut$ is a product of discretized terms $\mathcal{U}_{n'}^{\alpha,k'}$ with $n'<n$ and $k'\leq k$, we can invoke the induction hypothesis together with standard combinations of dispersive estimates, Hölder inequality and weighted Sobolev inequality. Note than even when rule \textbf{(C)} generates additional products of such terms, these can still be treated using the weighted Sobolev inequality  \ref{eq_weighted_Sobolev}, as in the proof of Proposition \ref{prop_unif_space_time_bounds_Un} (except in the three-dimensional case $d=3$ where this strategy can no longer apply, see Remark \ref{remark_NTS_d_3}).
\end{proof}

\section{Convergence of NQS} \label{sec:convergence_NQS}

The goal of this section is to prove Theorem \ref{theorem_convergence_NQS}, which establishes the convergence of our first numerical scheme \eqref{NQS}, now that we have all the necessary tools from Section \ref{section_dispersive_estimates} at our disposal. We begin by writing
\begin{equation} \label{eq_global_error_term_NQS}
\sup_{0 \leq j \leq J} \left\| u(t_j) - \sum_{n=0}^{N-1} \eps^n \mathfrak{U}_n^j \right\|_{L^2_x} \leq \left\| u(t_j) - \sum_{n=0}^{N-1} \eps^n U_n(t_j) \right\|_{\ell^{\infty}_{\tau} L^2_x} + \sum_{n=0}^{N-1} \eps^n \big\| U_n(t_j) - \mathfrak{U}_n^j \big\|_{\ell^{\infty}_{\tau} L^2_x}
\end{equation}
with the direct bound
\begin{equation} 
\Big\| u(t_j) - \sum_{n=0}^{N-1} \eps^n U_n(t_j) \Big\|_{\ell^{\infty}_{\tau} L^2_x} \leq \Big\| u(t_j) - \sum_{n=0}^{N-1} \eps^n U_n(t_j) \Big\|_{L^{\infty}_t \Sigma} \lesssim \eps^N
\end{equation}
from \cite{Carles2009} for the first term. For the second term, we can write for each term of the sum that
\begin{align}
\big\| U_n(t_j) - \mathfrak{U}_n^j \big\|_{\ell^{\infty}_{\tau} L^2_x}&\le \big\| U_n(t_j) - U_n^{\tau}(t_j) \big\|_{\ell^{\infty}_{\tau} L^2_x} + \big\| U_n^{\tau}(t_j) - \mathfrak{U}_n^j \big\|_{\ell^{\infty}_{\tau} L^2_x}\\ 
&\lesssim \tau^N\|\varphi\|_{\Sigma \cap H^{2N}_x} + \big\| U_n^{\tau}(t_j) - \mathfrak{U}_n^j \big\|_{\ell^{\infty}_{\tau} L^2_x}
\end{align}
using Proposition \ref{prop:Unstarbound} with the admissible pair $(q,r)=(\infty,2)$. To complete the proof, we now have to prove the following estimate
\begin{equation} \label{eq_induction_hypothesis_error_bound_Un_NQS}
\| U_n^{\tau}(t_j) - \mathfrak{U}_n^{j} \|_{\ell^q_{\tau} L^r_x} \leq C \tau^{N-n}
\end{equation}
for any admissible pairs $(q,r)$ and $0\le n\le 3$. While it would a priori be sufficient to consider only the admissible pair $(q,r)=(\infty,2)$, establishing stronger bounds for all admissible pairs is necessary to ensure the propagation of regularity. Moreover, this semi-discrete bound motivates the introduction of the auxiliary sequence $U_n^{\tau}$, as the continuous iterates $U_n(t_j)$ cannot be directly compared with the numerical scheme $\mathfrak{U}_n^{j}$ due to the lack of discrete-time dispersive estimates for~$U_n(t_j)$.

We prove the bounds \eqref{eq_induction_hypothesis_error_bound_Un_NQS} by induction on $n \geq 0$, focusing on the fourth-order case $N=4$ as previously announced, the lower-order cases $1 \leq N \leq 3$ being proven in a similar and easier way. For $n=0$, we simply remark by definition of \eqref{NQS} that
\[ U_0^{\tau}(t_\alpha)=S_{\tau}(t_\alpha)\varphi=\mathfrak{U}_0^{\alpha} \]
for all $\alpha \in \frac{\tau}{2}\IZ$, hence we directly get that
\[ \| U_0^{\tau}(t_j) - \mathfrak{U}_0^{j} \|_{\ell^q_{\tau} L^r_x} = 0 \]
which proves the initialization. We are now going to treat the cases $n\geq 1$, recalling that by definition of \eqref{NQS} we can write introducing the notation $\mathfrak{m}_n\coloneqq 3-n $ that
\[ \mathfrak{U}_n^j = - i \tau \sum_{\alpha=0}^{j-1} \sum_{\beta=0}^{\mathfrak{m}_n} S_{\tau}(t_j - t_{\alpha,\beta}^{(\mathfrak{m}_n)} ) \left(  \omega_{\beta}^{(\mathfrak{m}_n)} \mathfrak{F}_n^{\alpha+\frac{\beta}{\mathfrak{m}_n}} \right)  \]
with the conventions $t_{\alpha,0}^{(\mathfrak{m}_n)}=t_{\alpha}$, $\omega_{0}^{(0)}=1$ and $\frac{\beta}{\mathfrak{m}_n} = 0$ if $\mathfrak{m}_n=\beta=0$ (so for $n=3$). We can then write the global error
\begin{align}
U_n^{\tau}(t_j) - \mathfrak{U}_n^j & =   -i \int_0^{t_j} S_{\tau}(t_j-s) \ F_n^{\tau}(s)  \drm s +i \tau \sum_{\alpha=0}^{j-1} \sum_{\beta=0}^{\mathfrak{m}_n}  \omega_\beta^{(\mathfrak{m}_n)} S_{\tau}(t_j-t_{\alpha,\beta}^{(\mathfrak{m}_n)} ) \big(  F_n^{\tau}(t_{\alpha,\beta}^{(\mathfrak{m}_n)}) \big) \\
& \quad -i \tau \sum_{\alpha=0}^{j-1} \sum_{\beta=0}^{\mathfrak{m}_n}  S_{\tau}(t_j-t_{\alpha,\beta}^{(\mathfrak{m}_n)} ) \left( \omega_\beta^{(\mathfrak{m}_n)} \left( F_n^{\tau}(t_{\alpha,\beta}^{(\mathfrak{m}_n)}) - \mathfrak{F}_n^{\alpha+\frac{\beta}{\mathfrak{m}_n}} \right) \right)  \\
& \eqqcolon -i \mathfrak{B}_n^{j} - i \mathfrak{C}_n^{j}
\end{align}
coming respectively from the discretization of the time integral and the propagation of errors. For the first error term $B_n^{j}$, which corresponds to the high order Newton-Cotes error term, we write using Peano's error representation (see for instance \cite[Section 3.2]{StoerBulirsch2002}) associated to the Newton-Cotes formula that
\[  \mathfrak{B}_n^j= \frac{\tau}{2^{\mathfrak{m}_n}!} \sum_{\alpha=0}^{j-1} \int_{t_{\alpha}}^{t_{\alpha+1}} K_{2^{\mathfrak{m}_n}}^{\alpha}(s) \partial_s^{2^{\mathfrak{m}_n}} \left(  S_{\tau}(t_j-s) F^{\tau}_n(s) \right) \drm s,   \]
where $K_{2^{\mathfrak{m}_n}}^{\alpha}$ is the Peano kernel associated with the Newton-Cotes quadrature of order $2^{m_n}$ on the interval $\left[ t_{\alpha}, t_{\alpha+1} \right[$, which stands as a regular function satisfying $|K_{2^{\mathfrak{m}_n}}^{\alpha}(s)| \lesssim \tau^{2^{\mathfrak{m}_n}}$ uniformly in $\alpha$ for all $s \in \left[ t_{\alpha}, t_{\alpha+1} \right[$. For instance for the trapezoidal rule and Simpson's rule we respectively have
\[ K_{2}^{\alpha}(s)= \frac{(t_{\alpha+1}-s)(t_{\alpha}-s)}{2} \quad \text{and} \quad K_{4}^{\alpha}= \frac{(t_{\alpha+1}-s)^4}{4} - \frac{\tau}{6} \left(  4 \left( \frac{t_{\alpha}+ t_{\alpha+1}}{2} - s \right)^3_+ + (t_{\alpha+1}-s)^3 \right).  \]
We then piece-wisely define for $0 \leq s \leq t_j$ the function
\[K_{2^{\mathfrak{m}_n}}(s) \coloneqq \sum_{\alpha=0}^{j-1} K_{2^{\mathfrak{m}_n}}^{\alpha}(s) \mathds{1}_{\left[ t_{\alpha}, t_{\alpha+1} \right[}(s), \]
which still satisfies the bound $|K_{2^{\mathfrak{m}_n}}(s)| \lesssim \tau^{2^{\mathfrak{m}_n}}$ for all $0 \leq s \leq t_j$, uniformly in $T$. By direct differentiation we now remark that
\[  \partial_s( S_{\tau} (t_j - s) F_n^{\tau}(s)) = -i S_{\tau} (t_j - s) \Delta F_n^{\tau}(s) + S_{\tau} (t_j - s) \partial_s F_n^{\tau}(s),   \]
and $\partial_s F_n$ can be expressed as a combination of sum and product of $U_k^{\tau}$ for $0 \leq k \leq n-1$ as well as their derivatives in space using equation \eqref{eq_Un}. From this remark we define $\Lambda_n$ such that
\[ \partial_s^{2^{\mathfrak{m}_n}+1} ( S_{\tau}(t_j-s) \ F_n^{\tau}(s))  \eqqcolon S_{\tau} (t_j - s) \Lambda_n(s).   \]
We then get from Lemma \ref{lemma_discrete_continuous_Strichartz} that
\begin{align*}
 \| \mathfrak{B}_n^j \|_{\ell_{\tau}^q L^r_x} & = \left\| \int_0^{t_j} S_{\tau}(t_j-s) \left( K_{2^{m_n}}(s) \Lambda_n(s) \right) \drm s  \right\|_{\ell_{\tau}^q L^r_x}  \lesssim \|  K_{2^{m_n}} \Lambda_n \|_{L^{q'}_t L^{r'}_x} \\
 & \lesssim \tau^{2^{m_n}} \| \Lambda_n \|_{L^{q'}_t L^{r'}_x}
 \end{align*}
uniformly in $T$, which gives the result as $\tau^{2^{\mathfrak{m}_n}} \leq \tau^{N-n}$ from our choice of $\mathfrak{m}_n$. In order to estimate products of $U_k$ or their spatial derivatives when applying Hölder's inequality to control $\|\Lambda_n \|_{L^{q'}_t L^{r'}_x}$, we rely on Proposition~\ref{prop_unif_space_time_bounds_Un^tau}, which provides the following bounds 
\[ \| U_k^{\tau} \|_{L^{\infty}_t W^{\sigma,\infty}_x} \lesssim \| U_k^{\tau} \|_{L^{\infty}_t H^{\sigma+2}_x} \leq C   \]
for a generic constant $C>0$ uniform in time $T$, thanks to usual Sobolev embeddings as $d\leq 3$ and for any $\sigma \geq 0$.

\smallskip

We now turn to the analysis of the second error term $\mathfrak{C}_n^{j}$, namely the inductive propagation of error. For $n=1$ we simply remark that
\[  F_1^{\tau}(t_{\alpha})-\mathfrak{F}_1^{\alpha} = |U_0^{\tau}(t_{\alpha})|^{p-1}U_0^{\tau}(t_{\alpha}) - |\mathfrak{U}_0^{\alpha}|^{p-1}\mathfrak{U}_0^{\alpha}  = 0 \]
for all $\alpha \in \frac{\tau}{2}\IZ$, so $\mathfrak{C}_1^{j}=0$. For $n=2$, we can rewrite that
\[ \mathfrak{C}_2^{j}= \frac{\tau}{2} \sum_{\alpha}^{j-1} S_{\tau}(t_j-t_{\alpha})\left( F_{2}^{\tau}(t_{\alpha})-\mathfrak{F}_2^{\alpha}  \right) + \frac{\tau}{2} \sum_{\alpha}^{j-1} S_{\tau}(t_j-t_{\alpha+1})\left( F_{2}^{\tau}(t_{\alpha+1})-\mathfrak{F}_2^{\alpha+1}  \right),  \]
so applying discrete Sitrchartz estimates from Lemma \ref{lemma_discrete_Strichartz} to both terms we are brought back to estimate
\begin{align*}
\| F_{2}^{\tau}(t_{\alpha})-\mathfrak{F}_2^{\alpha} \|_{\ell_{\tau}^{q_2'} L^{r_2'}_x} = &\ \frac{p+1}{2} \|  |U_0^{\tau}(t_{\alpha})|^{p-1} (U_1^{\tau}(t_{\alpha})-\mathfrak{U}_1^{\alpha})\|_{\ell_{\tau}^{q_2'} L^{r_2'}_x} \\
& + \frac{p-1}{2} \|  |U_0^{\tau}(t_{\alpha})|^{p-3} (U_0^{\tau}(t_{\alpha}))^2 (\overline{U_1^{\tau}(t_{\alpha})}-\overline{\mathfrak{U}_1^{\alpha}})\|_{\ell_{\tau}^{q_2'} L^{r_2'}_x}  
\end{align*} 
as $U_0^{\tau}(t_\alpha)=\mathfrak{U}_0^{\alpha}$, thus by Hölder inequality we infer
\[ \| F_{2}^{\tau}(t_{\alpha})-\mathfrak{F}_2^{\alpha} \|_{\ell_{\tau}^{q_2'} L^{r_2'}_x} \lesssim \| U_0^{\tau}(t_{\alpha}) \|_{\ell^{\gamma}_{\tau} L^r_x}^{p-1} \| U_1^{\tau}(t_{\alpha})-\mathfrak{U}_1^{\alpha} \|_{\ell_{\tau}^q L^r_x}, \]
with $(q,r)$ admissible and $\gamma$ defined as in the proof of Proposition \ref{prop_unif_space_time_bounds_Un}. One can then conclude by estimating the $U_0^{\tau}(t_{\alpha})$ terms thanks to the weighted Sobolev inequality \eqref{eq_weighted_Sobolev} (as in the proof of Proposition \ref{prop_unif_space_time_bounds_Un}) alongside Proposition \ref{prop_unif_space_time_bounds_Un^tau}, and the $U_1^{\tau}(t_{\alpha})-\mathfrak{U}_1^{\alpha}$ by induction hypothesis. The case~$n=3$ is handled similarly, as
\begin{multline*}
\| \mathfrak{C}_3^{j} \|_{\ell_{\tau}^q L^r_x} = \left\| \tau \sum_{\alpha=0}^{j-1} S_{\tau}(t_j-t_{\alpha} ) \left(F_3^{\tau}(t_{\alpha})-\mathfrak{F}_3^{\alpha,\beta} \right) \right\|_{\ell^q_{\tau} L^r_x} \\
\lesssim \sum_{n_1+\ldots+n_p=n-1}\sum_{K=1}^p \left( \prod_{l=1}^{K-1} \|U_{n_l}^{\tau}(t_{\alpha}) \|_{\ell^{\gamma}_{\tau} L^{r}_x} \right) \| U_{n_K}^{\tau}(t_{\alpha}) - \mathfrak{U}_{n_K}^{\alpha} \|_{\ell^{q}_{\tau} L^{r}_x} \left( \prod_{l'=K+1}^{p} \|\mathfrak{U}_{n_l'}^{\alpha} \|_{\ell^{\gamma}_{\tau} L^{r}_x} \right) 
\end{multline*}
by discrete Strichartz estimates from Lemma \ref{lemma_discrete_Strichartz} and Hölder inequality. The terms 
$U_{n_l}^{\tau}(t_{\alpha})$ and~$\mathfrak{U}_{n_l'}^{\alpha}$ are then estimated through the weighted Sobolev inequality \eqref{eq_weighted_Sobolev} alongside Proposition \ref{prop_unif_space_time_bounds_Un^tau} and Proposition \ref{prop_unif_discrete_space_time_bounds_NQS}, while the terms $U_{n_K}^{\tau}(t_{\alpha}) - \mathfrak{U}_{n_K}^{\alpha}$ provide the needed accuracy in $\mathcal{O}(\tau)$ uniformly in time thanks to the induction argument, which ends the proof.

\section{Convergence of NTS} \label{sec:convergence_NTS}

The content of this Section is the proof of Theorem \ref{theorem_convergence_NTS}. First, as for \eqref{NQS} we write
\begin{equation} \label{eq_global_error_term_NQS}
\sup_{0 \leq j \leq J} \left\| u(t_j) - \sum_{n=0}^{N-1} \eps^n \mathcal{U}_n^j \right\|_{L^2_x} \lesssim \eps^N + \tau^{N} + \sum_{n=0}^{N-1} \eps^n \big\| U_n^{\tau}(t_j) - \mathcal{U}_n^j \big\|_{L^2_x}
\end{equation}
and we now prove by induction on $n\ge0$ the bound
\begin{equation} \label{eq_induction_hypothesis_error_bound_Un_NQS}
\mathcal{E}_n^{j,k}\coloneqq \|\nabla^kU_n^*(t_j) - U_n^{j,k} \|_{\ell^q_{\tau} L^r_x}\leq C \tau^{N-n-1-\lfloor\frac{k}{2}\rfloor}
\end{equation}
for any admissible pairs $(q,r)$, $0\le n\le N-1$, $0\le k\le2(N-n-2)$ and $0\le j\le J$, see condition~\eqref{ConditionNestedErrors}, which will eventually complete the proof. For $n=0$, we have
\begin{equation}
\nabla^k U_0^{\tau}(t_j)=S_{\tau}(t_j)\nabla^k\varphi
\end{equation}
for $0\le j\le J$ and $0\le k\le 2(N-2)$ thus $\mathcal{E}_0^{j,k}=0$. For $n\ge1$, we have
\begin{align*}
\nabla^kU_n^{\tau}(t_j)&=-i\int_0^{t_j}S_\tau(t_j-s)\nabla^kF_n^{\tau}(s)\drm s\\
&=-i\sum_{\alpha=0}^{j-1}\sum_{a\in\CT_n^{0,k}}c_k(a)\int_{t_\alpha}^{t_{\alpha+1}}S_\tau(t_j-s)a^\cut(s)\drm s
\end{align*}
where decorated trees are to be interpreted with truncated semigroup $S_\tau$, although we omit this dependence in the notation for conciseness. For the numerical scheme, we have
\begin{equation}
\mathcal{U}_n^{j,k}\coloneqq-i\sum_{\alpha=0}^{j-1}\sum_{\beta=0}^{m_n^k}\sum_{a\in\CT_n^{\beta,k}}c_k(a)\frac{\tau^{\beta+1}}{(\beta+1)!}S_\tau(t_j-t_\alpha)a_\alpha^\cut
\end{equation}
for $0\le j\le J$ and $0\le k\le 2(N-n-2)$ thus we split the error in two terms with
\begin{align}
\mathcal{E}_n^{k,j} & =-i\sum_{\alpha=0}^{j-1}\sum_{a\in\CT_n^{0,k}}c_k(a)\Big(\int_{t_\alpha}^{t_{\alpha+1}}S_\tau(t_j-s)a^\cut(s)\drm s-\sum_{\beta=0}^{m_n^k}\frac{\tau^\beta}{\beta!}\partial_s^\beta\big(S_\tau(t_j-s)a^\cut(s)\big)(t_\alpha)\Big)\\ 
&\quad-i\sum_{\alpha=0}^{j-1}\sum_{\beta=0}^{m_n^k}\sum_{a\in\CT_n^{\beta,k}}c_k(a)\frac{\tau^\beta}{\beta!}S_\tau(t_j-t_\alpha)\big(a^\cut(t_\alpha)-a_\alpha^\cut\big)\\ 
&\eqqcolon\mathcal{B}_n^{k,j}+\mathcal{C}_n^{k,j}
\end{align}
using the tree representation of $\nabla^kF_n$ and its time derivatives. The two error terms comes respectively from the discretization of the time integral and the propagation of errors. For the first term, we have
\begin{align}
\mathcal{B}_n^{k,j}&=-i\sum_{\alpha=0}^{j-1}\sum_{a\in\CT_n^{m_n^k+1,k}}c_k(a)\int_{t_\alpha}^{t_{\alpha+1}}\int_{t_\alpha}^{s_1}\frac{(s_1-s_2)^{m_n^k}}{m_n^k!}S_\tau(t_j-s_2)a^t(s_2)\drm s_2\drm s_1\\
&=-i\sum_{a\in\CT_n^{m_n^k+1,k}} \sum_{\alpha=0}^{j-1} c_k(a)\int_{t_\alpha}^{t_{\alpha+1}}S_\tau(t_j-s_2)a^t(s_2)\Big(\int_{s_2}^{t_{\alpha+1}}\frac{(s_1-s_2)^{m_n^k}}{m_n^k!}\drm s_1\Big)\drm s_2\\
&=-i\sum_{a\in\CT_n^{m_n^k+1,k}}c_k(a)\int_{0}^{t_j}S_\tau(t_j-s_2)a^t(s_2)\frac{(t_{\alpha+1}-s_2)^{m_n^k+1}}{(m_n^k+1)!}\drm s_2
\end{align}
using Taylor expansion with explicit remainder. Using discrete Strichartz estimates from Lemma~\ref{lemma_discrete_continuous_Strichartz}, we get
\begin{equation}
\|\mathcal{B}_n^{k,j}\|_{\ell_\tau^qL_x^r}\lesssim\tau^{N-n-1-\lfloor\frac{k}{2}\rfloor}\sum_{a\in\CT_n^{m_n^k+1,k}}\|a^t\|_{L_t^{q_2'}L_x^{r_2'}}
\end{equation}
for any admissible pair $(q_2,r_2)$ since $m_n^k=N-n-\lfloor\frac{k}{2}\rfloor-2$. For any $a\in\CT_n^{m_n^k+1,k}$, there exists decorated trees $b_1,\ldots,b_m$ with $m=p+q(p-1)$ and an integer $q\ge0$ such that
\begin{equation}
a^\cut=\nabla^{k_1}b_1\ldots\nabla^{k_m} b_m
\end{equation}
where $b_e\in\{a_e,\overline{a_e}\}$ with $a_e\in\CT_{n_e}$ with $n_1+\ldots+n_m=n$ and $2(\beta-q)+k=k_1+\ldots+k_m$, which corresponds to the representation \eqref{eq:arbrecut}. We then get
\[\|a^\cut\|_{L_t^{q_2'}L_x^{r_2'}}=\|\nabla^{k_1}b_1\ldots\nabla^{k_m}b_m\|_{L_t^{q_2'}L_x^{r_2'}} 
\lesssim \|\nabla^{k_1}b_1\|_{L_t^{q}L_x^{r}} \prod_{e=2}^m\|\nabla^{k_e}b_e\|_{L_t^{\gamma}L_x^{r}} \]
using Hölder inequality, with $\gamma$ and $r$ defined as in the proof of Proposition \ref{prop_unif_space_time_bounds_Un}. While the first term is bounded by Proposition \ref{prop_unif_space_time_bounds_Un^tau}, for the other terms we write that
\[  \|\nabla^{k_e}b_e\|_{L_x^{r}} \lesssim \frac{1}{|t|^{\delta}}  \|\nabla^{k_e}b_e\|_{L_x^{2}}  \|J(t)\nabla^{k_e}b_e\|_{L_x^{2}} \]
using the weighted Sobolev inequality \eqref{eq_weighted_Sobolev} with $\delta>0$ as in the proof of Proposition \ref{prop_unif_space_time_bounds_Un}. As~$\gamma \delta>1$, one can integrate in time, and by induction we are left by estimating thanks to commutators identities that
\[ \|J(t)\nabla^{k_e}b_e\|_{L^{\infty}_t L_x^{2}} \lesssim \|\nabla^{k_e} J(t)b_e\|_{L^{\infty}_t L_x^{2}} + \| \nabla^{k_e}b_e\|_{L^{\infty}_t L_x^{2}}, \]
which are bounded by Proposition \ref{prop_unif_space_time_bounds_Un^tau}. For the second error term, write again $a\in\CT_n^{\beta,k}$ as
\begin{equation}
a^\cut=\nabla^{k_1}b_1\ldots\nabla^{k_m}b_m
\end{equation}
where $b_i\in\{a_i,\overline{a_i}\}$ with $a_i\in\CT_{n_i}$ and $n_1+\ldots+n_m=n$ and $2(\beta-q)+k=k_1+\ldots+k_m$. We get
\begin{multline*}
\frac{\tau^\beta}{\beta!}S_\tau(t_j-t_\alpha)\big(a^\cut(t_\alpha)-a_\alpha^\cut\big) \\
=\frac{\tau^\beta}{\beta!}S_\tau(t_j-t_\alpha)\Big(\prod_{e=1}^m\nabla^{k_e}b_e(t_\alpha)-\prod_{e=1}^m\nabla^{k_e}(b_e)_\alpha\Big)\\
=\frac{\tau^\beta}{\beta!}S_\tau(t_j-t_\alpha)\sum_{e=1}^m\big(\nabla^{k_e}b_e(t_\alpha)-(\nabla^{k_e}b_e)_\alpha\big)\prod_{e'<e}\nabla^{k_{e'}}b_{e'}(t_\alpha)\prod_{e'>e}(\nabla^{k_{e'}}b_{e'})_\alpha
\end{multline*}
hence
\begin{multline*}
\|\sum_{\alpha=0}^j\frac{\tau^\beta}{\beta!}S_\tau(t_j-t_\alpha)\big(a^\cut(t_\alpha)-a_\alpha^\cut\big)\|_{\ell_\tau^qL_x^r}\\
\lesssim\tau^\beta\sum_{e=1}^m\|\big(\nabla^{k_e}b_e(t_\alpha)-(\nabla^{k_e}b_e)_\alpha\big)\prod_{e'<e}\nabla^{k_{e'}}b_{e'}(t_\alpha)\prod_{e'>e}(\nabla^{k_{e'}}b_{e'})_\alpha\|_{\ell_\tau^{q_2'}L_x^{r_2'}}\\ 
\lesssim\tau^\beta\sum_{e=1}^m\|\nabla^{k_e}b_e(t_\alpha)-(\nabla^{k_e}b_e)_\alpha\|_{\ell_\tau^{q}L_x^{r}}\prod_{e'<e}\|\nabla^{k_{e'}}b_{e'}(t_\alpha)\|_{\ell_\tau^{\gamma}L_x^{r}}\prod_{e'>e}\|(\nabla^{k_{e'}}b_{e'})_\alpha\|_{\ell_\tau^{\gamma}L_x^{r}}
\end{multline*}
for any admissible pairs $(q_2,r_2)$ using discrete Strichartz estimates from Lemma \ref{lemma_discrete_Strichartz}. We handle the first term by induction, while other terms are treated as before (with the use of discrete dispersive bounds from Proposition \ref{prop_unif_discrete_space_time_bounds_NTS} for the terms with $e'>e$). We finally get that
\begin{equation}
\|\sum_{\alpha=0}^j\frac{\tau^\beta}{\beta!}S_\tau(t_j-t_\alpha)\big(a^\cut(t_\alpha)-a_\alpha^\cut\big)\|_{\ell_\tau^qL_x^r}\lesssim\tau^\beta\sum_{e=1}^m\tau^{N-n_e-1-\lfloor\frac{k_e}{2}\rfloor}
\end{equation}
by induction hypotheses since $n_e<n$. As $q\ge0$, we have $k_e\le 2\beta+k$ thus
\begin{equation}
\|\sum_{\alpha=0}^j\frac{\tau^\beta}{\beta!}S_\tau(t_j-t_\alpha)\big(a^\cut(t_\alpha)-a_\alpha^\cut\big)\|_{\ell_\tau^qL_x^r}\lesssim\tau^{N-n-1-\lfloor\frac{k}{2}\rfloor}
\end{equation}
using also that $n_e<n$, which concludes the proof.

\section{Numerical experiments} \label{sec:numerical_experiments}

\subsection{The quintic case in dimension 1}

In this section, we illustrate Theorem \ref{theorem_convergence_NQS}, and we compare our newly developed multiscale approximation scheme with the classical Lie splitting from \cite{Carles2024}. In these numerical simulations, we use a standard Fourier pseudospectral method for space discretization with largest Fourier mode $K=2^9$. Our computations are carried out on a finite interval $\IT_a=\left[-\frac{\pi}{a},\frac{\pi}{a} \right[$ where~$a>0$ is a constant that will be specified along each numerical simulations. This constant is chosen to avoid finite-box size effects in our simulations, such as unwanted reflections or artificial transmissions due to periodic boundary conditions.

We first restrict our attention to the one-dimensional case $d=1$ (making simulations both easier and shorter to compute) for equation \eqref{NLS} in the quintic case $p=5$, which is covered by Theorem \ref{theorem_convergence_NQS}. We will take $N=3$, which guarantees a convergence in $\mathcal{O}(\eps^3)$ of our multiscale scheme towards the true solution as the parameter $\eps$ tends to 0 under the CFL condition $\tau \leq \eps$. With these parameters, \eqref{NQS} writes as 
\[ \left|
\begin{aligned} 
& \  \mathfrak{U}_0^{j}= S_{\tau}(t_{j}) \varphi,   \\
& \  \mathfrak{U}_1^{j}= - i \frac{\tau}{2} \left( \sum_{\alpha=0}^{j-1} \sum_{\beta=0}^{1} S_{\tau} (t_j-t_{\alpha+\beta}) \left(|\mathfrak{U}_0^{\alpha+\beta} |^4 \mathfrak{U}_0^{\alpha+\beta}  \right)  \right),  \\
& \  \mathfrak{U}_2^{j}= - i \tau  \sum_{\alpha=0}^{j-1} S_{\tau} (t_j-t_\alpha) \left( 3|\mathfrak{U}_0^{\alpha} |^4 \mathfrak{U}_1^{\alpha} + 2|\mathfrak{U}_0^{\alpha} |^2 \left( \mathfrak{U}_0^{\alpha}\right)^2 \overline{\mathfrak{U}_1^{\alpha}}  \right)  ,
\end{aligned}
 \right. \]

where we recall that $t_j=j \tau$ for $0 \leq j \leq J$ with $T=J\tau$.

\begin{remark}
For computational efficiency in both memory and runtime, we rather compute~$\mathfrak{V}_2^j =S_{\tau}(-t_j) \mathfrak{U}_2^j$ using the recursive formula:
\[  \mathfrak{V}_2^{j+1}= \mathfrak{V}_2^{j} - i \tau S_{\tau} (-t_j) \left( 3|\mathfrak{U}_0^{j} |^4 \mathfrak{U}_1^{j} + 2|\mathfrak{U}_0^{j} |^2 \left( \mathfrak{U}_0^{j}\right)^2 \overline{\mathfrak{U}_1^{j}}  \right).  \]
We only revert to $\mathfrak{U}_2^j=S_{\tau}(t_j) \mathfrak{V}_2^j$ at specific times of interest.
\end{remark}

\subsubsection{Numerical accuracy and propagation of Gaussian data for quintic 1D NLS}

To the best of the authors' knowledge, no explicit formula exists for non-trivial solutions of equation~\eqref{NLS} in the defocusing case. This makes it challenging to obtain a reference solution for comparison with our scheme. We avoid these technicalities by comparing our numerical scheme directly to the solutions of the linear Schrödinger sub-problems
\[ \left|
\begin{aligned} 
& \ i \partial_t U_0 +\Delta U_0 =0, & \quad U_0(0)=\varphi, \\
& \ i \partial_t U_1 +\Delta U_1 =|U_0|^4 U_0, & \quad U_1(0)=0, \\
& \ i \partial_t U_2 +\Delta U_2 =3 |U_0|^4 U_1 + 2 |U_0|^2 U_0^2 \overline{U_1}, & \quad U_2(0)=0,
\end{aligned}
 \right. \]
 rather than to the general solution $u$ of equation \eqref{NLS}. Indeed, explicit formulas for $U_0$, $U_1$ and $U_2$ can be derived by taking the initial condition $\varphi$ as a Gaussian function. This computations rely on the well-known property that Gaussians functions remains Gaussians under the Schrödinger flow: for instance on $\IR$, for any $z\in \IC$ and $f(x)=e^{-z x^2}$, we have
\[ \left(e^{i t \Delta} f \right)(x)= \frac{1}{\sqrt{1+4izt}} e^{- \frac{z}{1+4izt}x^2}. \]
Let then $\varphi(x) =e^{-\frac{x^2}{2}}$ be our initial condition. Applying recursively the previous formula we then get that
\[  U_0(t,x)=\frac{1}{\sqrt{\lambda(t)}} e^{-\frac{1}{2 \lambda (t)} x^2} \]
with $\lambda(t)= 1+2it$,
\[ U_1(t,x) = -i \int_0^t \frac{1}{|\lambda(s)|^2\sqrt{\lambda(s)}} \frac{1}{\sqrt{1+4i(t-s)z(s) }} e^{\frac{z(s)}{1+4i(t-s)z(s)}x^2} \drm s \] 
with $z(s)=\frac{2}{1+4s^2}+\frac{1}{2 \lambda(s)}$, and 
\begin{align*}
U_2(t)  = & - 3 \int_0^t \frac{1}{|\lambda(s)|^2} \int_{0}^s \frac{1}{|\lambda(r)|^2\sqrt{\lambda(r)} \sqrt{\Theta(r,s)}} \frac{1}{\sqrt{1+4i(t-s) \zeta(r,s)}} e^{- \frac{\zeta(r,s)}{1+4i \zeta(r,s) (t-s)} x^2} \drm r \drm s \\
& +2 \int_0^t \frac{1}{|\lambda(s)|\lambda(s)} \int_{0}^s \frac{1}{|\lambda(r)|^2\sqrt{\overline{\lambda(r)}}\sqrt{\overline{\Theta(r,s)}}} \frac{1}{\sqrt{1+4i(t-s) \widetilde{\zeta(r,s)}(r,s)}} e^{- \frac{\widetilde{\zeta(r,s)}}{1+4i \widetilde{\zeta(r,s)} (t-s)} x^2} \drm r \drm s
\end{align*}  
with $\Theta(r,s)=1+4i(s-r)z(r)$ as well as 
\[ \zeta(r,s)= \frac{z(r)}{\Theta(r,s)} + \frac{2}{1+4s^2} \quad \text{and} \quad \widetilde{\zeta(r,s)} = \frac{\overline{z(r)}}{\overline{\Theta(r,s)}} + \frac{1}{1+4s^2} + \frac{1}{\lambda(s)}. \]

We first fix $T=1$ and $a=0.05$, and we compute $(U_n(T))_{0 \leq n \leq 2}$ using the above exact formulas discretized by a rectangle rule with a very precise stepsize $\tau_0 = 1.10^{-4}$ to get a reference solution. We then compute the numerical errors $e_n \coloneqq \|U_n(T) - \mathfrak{U}_n^{J} \|_{L^2(\IR)}$ for $n=0,1$ and $2$ as a function of the time step $\tau$ in the left part of Figure 1. As expected, we observe that $\mathfrak{U}_2^J$ provides a first-order approximation of $U_2$, while $\mathfrak{U}_1^J$ achieves a second-order approximation of  $U_1$. We also note that the error for $U_0$ is negligible, as it is computed using a direct explicit formula.

\begin{figure}[h]
	\centering
	\captionsetup{width=0.75\textwidth}
		\includegraphics[width=0.80\textwidth,trim = 15cm 10cm 15cm 10cm, clip]{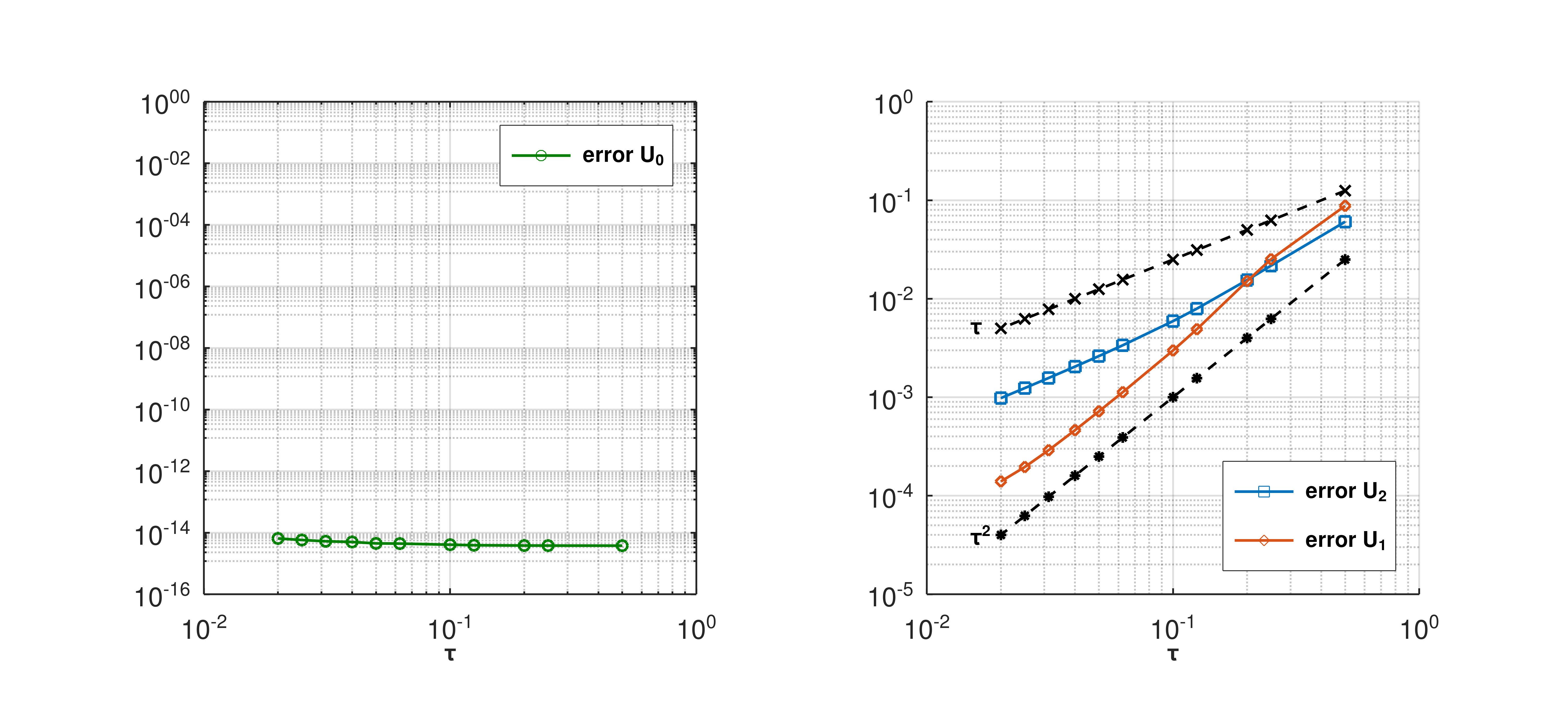}	
	\caption{Plot of the $L^2$-numerical errors of $\mathfrak{U}_0^J$ (left panel) and of $\mathfrak{U}_1^J$ and $\mathfrak{U}_2^J$ (right panel) as time step $\tau \rightarrow 0$, with fixed final time $T=1$.}
	\label{fig:convergence_Un.png}
\end{figure}

\subsubsection{Convergence in the weakly nonlinear regime for quintic 1D NLS and comparison}

We now compare the evolution of the errors in $L^2$-norm produced both by our multiscale scheme and by the Lie splitting scheme described in \cite{Carles2024}, as the weakly nonlinear parameter $\eps$ goes to 0. These errors are evaluated against a reference solution computed using a highly accurate Strang splitting approximation (with $\tau_1=1.10^{-4}$). All computations are performed with final time $T=1$ and space constant $a=0.05$.

The usual operator splitting methods for the time integration of \eqref{NLS} are based on the solutions of the subproblems
\[ \left|
\begin{aligned} 
& \ i  \partial_t v= - \Delta v, & \quad v(0)=v_0,  \\   
& \ i  \partial_t w(t,x)= \eps |w|^{p-1} w , & \quad w(0)=w_0,
\end{aligned}
 \right. \]
and the associated operators are then explicitly given, for $t \in \IR$, by
\[  v(t)=S(t) v_0=e^{it \Delta} v_0, \]
\[  w(t)=\Phi^t_{\mathcal{N}}(w_0)= e^{-i \eps t |w_0|^{p-1}} w_0. \]
The Lie splitting scheme from \cite{Carles2024}, with projected linear flow $S_{\tau}$ instead of $S$, is then given by the recursive formula
\[  u_{LS}^{j+1}= S_{\tau}(\tau) \circ \Phi^{\tau}_{\mathcal{N}} (u^{j}_{LS}), \quad u^0_{LS}=\varphi,  \]
while the second-order Strang splitting writes as 
\[  u^{j+1}_{ST}= S_{\tau} \left(\frac{\tau}{2} \right) \circ \Phi^{\tau}_{\mathcal{N}} \circ S_{\tau}\left(\frac{\tau}{2} \right) (u^{j}_{ST}), \quad u^0_{ST}=\varphi.  \]

In Figure \ref{fig:convergence_order_3.png}, we compute the errors 
\[E_{LS}(\eps)=\| u_{ST}^J-u^{J}_{LS} \|_{L^2_x} \quad \text{and} \quad E_{NQS}(\eps)=\| u_{ST}^J-\mathfrak{U}_0^J - \eps \mathfrak{U}_1^J - \eps^2 \mathfrak{U}_2^J \|_{L^2_x}, \]
 evaluated as $\eps$ varies from $1$ to $1.10^{-4}$. In the left panel, the time step $\tau=0.01$ is used, while in the right panel, a smaller time step $\tau=0.001$ is employed.
 
\begin{figure}[h]
	\centering
	\captionsetup{width=0.75\textwidth}
		\includegraphics[width=0.80\textwidth,trim = 15cm 10cm 15cm 10cm, clip]{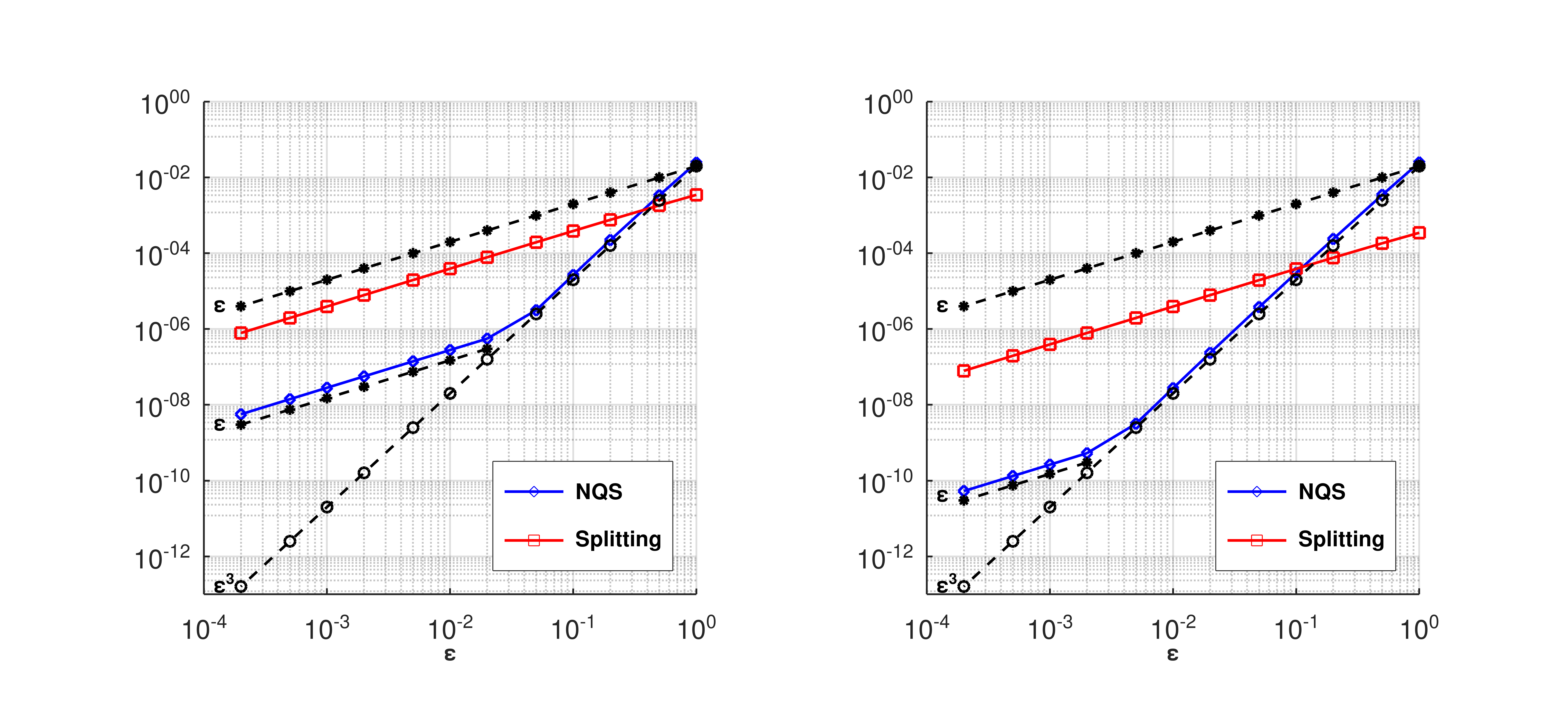}	
	\caption{Log-log plot of the convergence of the $L^2$-numerical errors for the splitting scheme $E_{LS}(\eps)$ and the multiscale scheme $E_{NQS}(\eps)$ as $\eps \rightarrow 0$, with fixed final time $T=1$ and time step $\tau=0.01$ (left pannel) or $\tau=0.001$ (right pannel).}
	\label{fig:convergence_order_3.png}
\end{figure}

As anticipated, our multiscale scheme exhibits third-order accuracy for $\eps=1$ to $\eps \simeq 2 \tau$. As $\eps$ approaches $\tau$,  the scheme gradually transitions to first-order accuracy, effectively illustrating the CFL condition $\tau \leq \eps$ and the results of Theorem \ref{theorem_convergence_NQS}. We also observe that our multiscale scheme surpasses the accuracy of the splitting scheme (which is first-order in $\eps$) only when $\eps \leq 0.1$.  This aligns with the fact that our analysis is specifically designed for the weakly nonlinear regime $\eps \rightarrow 0$. 

\subsection{The cubic case in dimension 2}

We now turn to a numerical illustration of Theorem \ref{theorem_convergence_NTS}, implementing \eqref{NTS} for the cubic \eqref{NLS} equation $p=3$ on the two dimensional setting $d=2$. As previously, all computations are performed using Fourier pseudospectral method for space discretization with $K=2^8$, on a finite interval $\IT_a^2=\left[-\frac{\pi}{a},\frac{\pi}{a} \right[^2$ where~$a>0$ is a constant chosen to avoid finite-box size effects that will be specified along numerical tests. 

\smallskip

We will now take a fourth order $N=4$, so that our scheme is expected to have a $\mathcal{O}(\eps^4)$ convergence towards the true solution as the parameter $\eps \rightarrow 0$ under the CFL condition $\tau \leq \eps$. We recall that $t_j=j \tau$ for $0 \leq j \leq J$ with $T=J\tau$, and with these parameters, our scheme recursively writes as 

\[ \left|
\begin{aligned} 
& \ \mathcal{U}_0^{j+1} = S_{\tau}(\tau)  \mathcal{U}_0^j, & \quad \mathcal{U}_0^0=\varphi, \\
& \ \mathcal{U}_0^{1,j+1} = S_{\tau}(\tau)   \mathcal{U}_0^{1,j}, & \quad  \mathcal{U}_0^{1,0}=\nabla \varphi, \\
& \ \mathcal{U}_0^{2,j+1} = S_{\tau}(\tau)   \mathcal{U}_0^{2,j}, & \quad  \mathcal{U}_0^{2,0}=\Delta \varphi, \\
& \ \mathcal{U}_0^{3,j+1} = S_{\tau}(\tau)   \mathcal{U}_0^{3,j}, & \quad  \mathcal{U}_0^{3,0}=\nabla \Delta \varphi, \\
& \ \mathcal{U}_0^{4,j+1} = S_{\tau}(\tau)   \mathcal{U}_0^{4,j}, & \quad  \mathcal{U}_0^{4,0}=\Delta^2 \varphi, \\
& \ \mathcal{V}_1^{j+1} = \mathcal{V}_1^j - i S_{\tau}(-t_j) \left[ \tau |\mathcal{U}_0^j|^2 \mathcal{U}_0^j - i \tau^2 \left( 2 |\mathcal{U}_0^{1,j}|^2 \mathcal{U}_0^j + (\mathcal{U}_0^{1,j})^2 \overline{\mathcal{U}_0^j} + (\mathcal{U}_0^j)^2 \overline{\mathcal{U}_0^{2,j}}\right. \right) & \\
& \quad \quad \quad \left. - 2\frac{\tau^3}{3} \left((\mathcal{U}_0^{2,j})^2 \overline{\mathcal{U}_0^j} + 4 \mathcal{U}_0^{2,j} |\mathcal{U}_0^{1,j}|^2 + 4 (\mathcal{U}_0^{1,j})^2 \overline{\mathcal{U}_0^{2,j}} + 2 |\mathcal{U}_0^{2,j}|^2 \mathcal{U}_0^j  \right. \right. &  \\
& \quad \quad \quad \left. \left. + \overline{\mathcal{U}_0^{4,j}} (\mathcal{U}_0^j)^2 + 4 \mathcal{U}_0^j \mathcal{U}_0{1,j}\cdot \overline{\mathcal{U}_0^{3,j}} \right)   \right] & \quad V_1^0=0, \\
& \ \mathcal{U}_1^{j+1} =S_{\tau}(t_{j+1}) \mathcal{V}_1^{j+1} & \quad \mathcal{U}_1^0=0, \\
& \mathcal{V}_1^{1,j+1} = \mathcal{V}_1^{1,j} - i S_{\tau}(-t_j) \left[ \tau \left( 2 |\mathcal{U}_0^j|^2 \mathcal{U}_0^{1,j} + (\mathcal{U}_0^j)^2 \overline{\mathcal{U}_0^{1,j}}  \right) \right] & \quad \mathcal{V}_1^{1,0}=0,\\
& \mathcal{U}_1^{1,j+1} = S_{\tau}(t_{j+1}) \mathcal{V}_1^{1,j+1} & \quad \mathcal{U}_1^{1,0}=0, \\
& \mathcal{V}_1^{2,j+1} = \mathcal{V}_1^{2,j} - i S_{\tau}(-t_j) \left[ \tau \left( 2 \mathcal{U}_0^{2,j} |\mathcal{U}_0^j|^2 + 4 |\mathcal{U}_0^{1,j}|^2 \mathcal{U}_0 + 2 (\mathcal{U}_0^{1,j})^2 \overline{\mathcal{U}_0^j}  + (\mathcal{U}_0^j)^2 \overline{\mathcal{U}_0^{2,j}} \right) \right] & \quad \mathcal{V}_1^{2,0}=0,\\
& \mathcal{U}_1^{2,j+1} = S_{\tau}(t_{j+1}) \mathcal{V}_1^{2,j+1} & \quad \mathcal{U}_1^{2,0}=0, \\
& \mathcal{V}_2^{j+1} = \mathcal{V}_2^j - i S_{\tau}(- t_j) \left[ \tau \left( 2 |\mathcal{U}_0^j|^2 \mathcal{U}_1^j + (\mathcal{U}_0^j)^2 \overline{\mathcal{U}_1^j} \right) - \frac{\tau^2}{2} \left(  |\mathcal{U}_0^j|^4 \mathcal{U}_0^j + 4 |\mathcal{U}_0^{1,j}|^2 \mathcal{U}_1^j \right. \right. & \\
& \quad \quad \quad  + 4 \overline{\mathcal{U}_0^{2,j}} \mathcal{U}_0^j \mathcal{U}_1^j + 4 \overline{\mathcal{U}_0^j} \mathcal{U}_0^{1,j} \cdot \mathcal{U}_1^{1,j} + 4 \mathcal{U}_0^j \overline{\mathcal{U}_0^{1,j}} \cdot \mathcal{U}_1^{1,j} + 2 \overline{\mathcal{U}_1^{2,j}} (\mathcal{U}_0^j)^2 \\
& \quad \quad \quad \left. \left. + 2 (\mathcal{U}_0^{1,j})^2 \cdot \overline{\mathcal{U}_1^{j}} + 4 \mathcal{U}_0^j \mathcal{U}_0^{1,j} \cdot \overline{\mathcal{U}_1^{1,j}} \right)   \right] & \quad \mathcal{V}_2^{0}=0,\\
& \mathcal{U}_2^{j+1}= S_{\tau}(t_{j+1}) \mathcal{V}_2^{j+1} & \quad \mathcal{U}_2^{0}=0, \\
& \mathcal{V}_3^{j+1} = \mathcal{V}_3^j - i \tau S_{\tau}(- t_j) \left( 2 |\mathcal{U}_0^j|^2 \mathcal{U}_2^j + (\mathcal{U}_0^j)^2 \overline{\mathcal{U}_2^j} + 2 |\mathcal{U}_1^j|^2 \mathcal{U}_0^j + (\mathcal{U}_1^j)^2 \overline{\mathcal{U}_0^j}  \right) & \quad \mathcal{V}_3^{0}=0,\\
& \mathcal{U}_3^{j+1} = S_{\tau}(t_{j+1}) \mathcal{V}_3^{j+1} & \quad \mathcal{U}_3^{0}=0. \\
\end{aligned}
 \right. \]
 
 \subsubsection{Numerical accuracy and propagation of Gaussian data for cubic 2D NLS}
 
 To check accuracy for approximation for $U_0$, $U_1$ and $U_2$, we use the propagation of 2D Gaussian data
  \[ \left(e^{i t \Delta} e^{-z |x|^2} \right)(x)= \frac{1}{1+4izt} e^{- \frac{z}{1+4izt} |x|^2} \quad \text{for} \ z \in \IR. \] 
For initial condition $\varphi(x)=e^{-|x|^2/2}$, this leads to 
  \[ U_0(t)=\frac{1}{\lambda(t)} e^{- \frac{1}{2 \lambda(t)} |x|^2}, \quad \nabla U_0(t) = - \frac{x}{\lambda(t)^2} e^{- \frac{1}{2 \lambda(t)} |x|^2}  \]
  and
  \[ \Delta U_0(t)= \frac{1}{\lambda(t)^2} \left( \frac{|x|^2}{\lambda(t)} - 2  \right) e^{- \frac{1}{2 \lambda(t)} |x|^2}  \]
  with $\lambda(t)=1+2it$. We also get
  \[ U_1(t)=  -i \int_0^t \frac{1}{|\lambda(s)|^2 \lambda(s)} \frac{1}{1+4i(t-s)\nu(s)} e^{-\frac{\nu(s)}{1+4i(t-s)\nu(s)} |x|^2} \drm s \]
  with $\nu(s)=\frac{1}{1+4s^2}+\frac{1}{2 \lambda(s)}$, as well as
   \[ \nabla U_1(t)=  2 i \int_0^t \frac{1}{|\lambda(s)|^2 \lambda(s)} \frac{\nu(s)}{(1+4i(t-s)\nu(s))^2} x e^{-\frac{\nu(s)}{1+4i(t-s)\nu(s)} |x|^2} \drm s \] 
and
\[ \Delta U_1(t)=  4 i \int_0^t \frac{1}{|\lambda(s)|^2 \lambda(s)} \frac{\nu(s)}{(1+4i(t-s)\nu(s))^2} \left( 1 - \frac{\nu(s) |x|^2}{1+4i(t-s)\nu(s)}  \right) e^{-\frac{\nu(s)}{1+4i(t-s)\nu(s)} |x|^2} \drm s. \]
Finally we compute
\begin{align*}
U_2(t)  = & -2 \int_0^t \frac{1}{|\lambda(s)|^2} \int_{0}^s \frac{1}{|\lambda(r)|^2 \lambda(r) \Upsilon(r,s)} \frac{1}{1+4i(t-s) \eta(r,s)} e^{- \frac{\mu(r,s)}{1+4i \mu(r,s) (t-s)} |x|^2} \drm r \drm s \\
& + \int_0^t \frac{1}{\lambda(s)^2} \int_{0}^s \frac{1}{|\lambda(r)|^2 \overline{\lambda(r)}\overline{\Upsilon(r,s)}} \frac{1}{\sqrt{1+4i(t-s) \widetilde{\mu(r,s)}(r,s)}} e^{- \frac{\widetilde{\mu(r,s)}}{1+4i \widetilde{\mu(r,s)} (t-s)} |x|^2} \drm r \drm s
\end{align*}  
with $\Upsilon(r,s)=1+4i(s-r)\nu(r)$ as well as 
\[ \mu(r,s)= \frac{\nu(r)}{\Upsilon(r,s)} + \frac{1}{1+4s^2} \quad \text{and} \quad \widetilde{\mu(r,s)} = \frac{\overline{\nu(r)}}{\overline{\Upsilon(r,s)}} +\frac{1}{\lambda(s)}. \]
   
\medskip
   
We now fix $T=1$ and $a=1/6$, and we compute $(U_n(T))_{0 \leq n \leq 2}$ using the above formulas discretized by a rectangle rule with a precise stepsize $\tau_0 = 1.10^{-3}$ to get a reference solution. We then compute the numerical errors $e_n \coloneqq \|U_n(T) - \mathcal{U}_n^{J} \|_{L^2(\IR^2)}$ for $n=0$ and  $n=1,2$ as a function of the time step $\tau$ in respectively the left and right part of Figure 1. We well observe that the errors for instance for $\mathcal{U}_0^J$, $\mathcal{U}_0^{2,J}$ and $\mathcal{U}_0^{4,J}$ are negligible, that $\mathcal{U}_1^{1,J}$ and $\mathcal{U}_1^{2,J}$ are well first-order convergent towards $\nabla U_1(T)$ and $\Delta U_1(T)$, and that $\mathcal{U}_2^J$ well provides a second-order approximation of $\mathcal{U}_2$. On the other hand, as we reach very high precision, $\mathcal{U}_1^J$ achieves an in-between second-order and third-order approximation of $U_1(T)$. Note that we do not investigate the first-order convergence of $U_3$, as it relies solely on a left-point rectangle rule and involves increasingly intricate Gaussian computations, making both the computations and simulations significantly more demanding.

\begin{figure}[h]
	\centering
	\captionsetup{width=0.75\textwidth}
		\includegraphics[width=0.80\textwidth,trim = 15cm 10cm 15cm 10cm, clip]{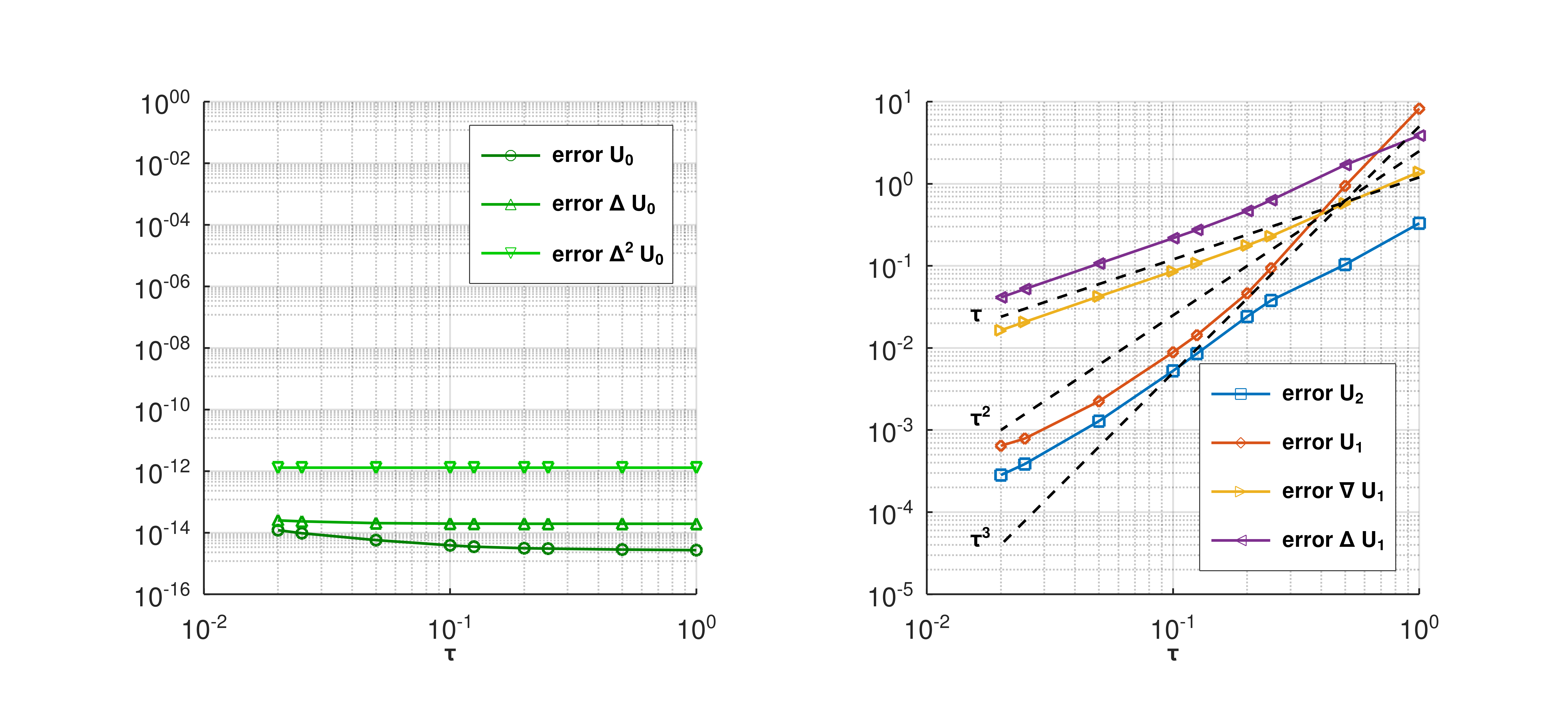}	
	\caption{ Log-log plot of the convergence of the $L^2$-numerical errors of $\mathcal{U}_0^J$ and derivatives (left) and $\mathcal{U}_1^J$ and $\mathcal{U}_2^J$ as well as derivatives (right) as time step $\tau \rightarrow 0$, with fixed final time $T=0.1$}
	\label{fig:comparison_gaussian_cubic_NTS.png}
\end{figure}
   
\subsubsection{Convergence in the weakly nonlinear regime for cubic 2D NLS}

We now turn to the error in $L^2$-norm produced by \eqref{NTS} scheme as $\eps$ goes to 0. Once again the reference solution is computed using a highly accurate Strang splitting approximation (with $\tau_1=1.10^{-4}$). All computations are performed with final time $T=1$ and space constant $a=1/3$. In Figure \ref{fig:NTSorder4.png}, we compute the errors 
\[ E_{NTS}(\eps)=\| u_{ST}^J-\mathcal{U}_0^J - \eps \mathcal{U}_1^J - \eps^2 \mathcal{U}_2^J - \eps^3 \mathcal{U}_3^J  \|_{L^2_x}, \]
 evaluated as $\eps$ varies from $1$ to $1.10^{-4}$. In the left panel, the time step $\tau=0.01$ is used, while in the right panel, a smaller time step $\tau=0.001$ is employed. As expected, \eqref{NTS} scheme exhibits fourth-order accuracy until the CFL condition $\tau \geq \eps$, which well confirms the result from Theorem~\ref{theorem_convergence_NTS}.
 
 \begin{figure}[h]
	\centering
	\captionsetup{width=0.75\textwidth}
		\includegraphics[width=0.80\textwidth,trim = 15cm 10cm 15cm 10cm, clip]{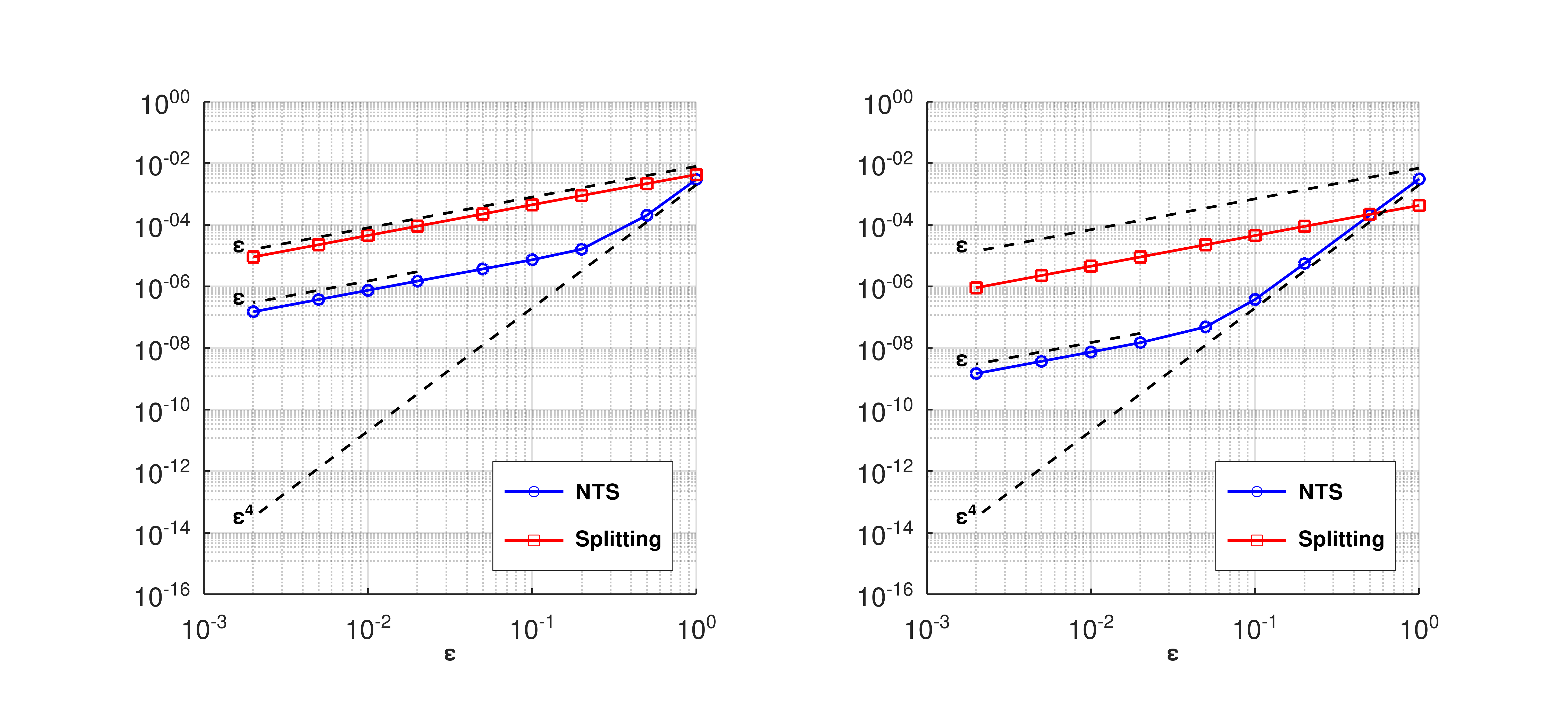}	
	\caption{Log-log plot of the convergence of the $L^2$-numerical errors for the splitting scheme $E_{LS}(\eps)$ and the multiscale scheme $E_{NTS}(\eps)$ as $\eps \rightarrow 0$, with fixed final time $T=1$ and time step $\tau=0.01$ (left panel) or $\tau=0.001$ (right panel).}
	\label{fig:NTSorder4.png}
\end{figure}
   
\section{Application to Wave Turbulence} \label{sec:wave_turbulence}

We conclude this paper by illustrating our numerical scheme within the physical framework that originally motivated our work, which is the theory of wave turbulence. We briefly recall the core concepts of the theory, focusing specifically on the two-dimensional cubic Schrödinger equation throughout this section, that is equation \eqref{NLS} with $p=3$ and $d=2$.

Rather than dealing with the full space $\IR^2$, we write our formal analysis on the large torus~$\IT^2_L \coloneqq L\mathbb{T}^2$ for some $L > 0$, which aligns with the numerical necessity of working in a bounded spatial domain. The initial data is assumed to satisfy the random phase (RP) condition
\[ \varphi(x)=\sum_{k \in \IZ^2_L} \widehat{\varphi}_k e^{ik\cdot x}, \quad \widehat{\varphi}_k=\sqrt{\phi(k)} e^{i \theta_k}  \]
where $\theta_k$ are independent variables, uniformly distributed on $\left[ 0, 2\pi \right]$ and $\phi$ is a regular function with fast decrease as $|k|\rightarrow + \infty$. We denote by
\[  u_k \coloneqq \widehat{u}(k)= \frac{1}{(2\pi L)^2} \int_{\IT^2_L} u(x) e^{-ik\cdot x} \drm x  \]
the $k$-th Fourier coefficient of $u$ with $k \in \IZ^2_L \coloneqq \frac{1}{L} \IZ^2$. Filtering by the linear Schrödinger flow~$v_k(t)=e^{it|k|^2}u_k(t)$, the unknowns $(v_k)_{k \in \IZ^2_L}$ then satisfy the following set of coupled nonlinear ODEs
\[  i\partial_t v_k = \eps \sum_{k+k_2=k_1+k_3} v_{k_1} \overline{v_{k_2}} v_{k_3} e^{-i t (|k|^2-|k_1|^2+|k_2|^2-|k_3|^2)}.\]

Wave turbulence theory predicts that in the large-volume limit $L \to \infty$ and the weakly nonlinear regime $\varepsilon \to 0$, the statistical behavior of the mean wave action density $n_k = \mathbb{E}[|v_k|^2]$ is governed, over a characteristic time scale known as the \textit{wave kinetic time} $T \sim \varepsilon^{-2}$, by the so-called \emph{wave kinetic equation} 

\[ \partial_t n_k = \int_{\substack{k+k_2=k_1+k_3 \\ |k|^2+|k_2|^2=|k_1|^2+|k_3|^2}}  n_{k} n_{k_1} n_{k_2} n_{k_3}  \left( \frac{1}{n_k} + \frac{1}{n_{k_2}}- \frac{1}{n_{k_1}} - \frac{1}{n_{k_3}}\right) \drm k_1 \drm k_2 \drm k_3 . \]

Physically, the evolution of $n_k$ reflects how energy is redistributed across scales: a direct cascade refers to energy transfer toward high wavenumbers (corresponding to smaller spatial scales), while an inverse cascade refers to particles transfer toward low wavenumbers (corresponding to larger spatial scales). These cascades are universal features of turbulent dynamics. In the framework of Wave Turbulence theory, they are characterized by the emergence of Kolmogorov-Zakharov (KZ) solutions, which are particular stationary solutions of the wave kinetic equation and take the form of power-law spectra
\[  n_k=k^{-\gamma} \]
at least within a certain range of frequencies known as the inertial range. Note that the exponent $\gamma$ encodes universal aspects of the system, depending only on its dimension and nonlinearity. In the case of the two-dimensional cubic Schrödinger equation, the theory predicts KZ spectra of the form
\[ n_k=k^{-\frac{4}{3}} \quad \text{(inverse cascade)} \quad  \text{and} \quad n_k=k^{-2} \quad \text{(direct cascade)}. \]
Returning to our perturbative expansion in powers of $\varepsilon$, the solution $u$ in frequency reads
\[ u_k=\widehat{U_0}(k)+\eps \widehat{U_1}(k)+\eps^2 \widehat{U_2}(k) + \ldots   \]
Taking products and expectations, we obtain
\[ \mathbb{E}[|u_k|^2]= \mathbb{E}[|\widehat{U_0}(k)|^2] + 2 \eps \Re  \mathbb{E}[\overline{\widehat{U_0}(k)} \widehat{U_1}(k)]   + \eps^2 \left( \mathbb{E}[|\widehat{U_1}(k)|^2] + 2 \Re \mathbb{E}[\overline{\widehat{U_0}(k)} \widehat{U_2}(k)]  \right) + \mathcal{O}(\eps^3)  . \]
The leading-order term in $\mathcal{O}(\varepsilon^0)$ corresponds to the linear dynamics, while the first-order contribution vanishes due to probabilistic cancellations. It is well known that nonlinear effects first arise at order~$\varepsilon^2$, marking the point where kinetic behavior becomes dominant and KZ-type solutions can emerge.
However, since our model includes neither forcing nor dissipation, we cannot expect KZ spectra to appear as stationary states (as in many physical setups), but rather as transient, intermediate states, as emphasized in \cite{Krstulovic2022} for the three-dimensional cubic Schrödinger equation.

We now present our numerical simulations. We fix the weak nonlinearity parameter at $\eps=0.1$, take $L=16\pi$, and discretize each spatial direction with $K=2^{10}$. The total simulation time is~$T=100$, with a time step $\tau=0.1$. Following the initial conditions used in \cite{Krstulovic2022}, we define the Fourier transform of the initial data as
\[ \phi(k) = \frac{1}{L^2 |k|^2} e^{-\frac{|k-k_s|^2}{\sigma^2}} \] 
with Gaussian mean $k_s=15$ and width $\sigma=1$ (note that we use the convention $|k|^2=1$ if $k=0$ in the above expression).

We use our \eqref{NQS} with $N=3$, computing $U_0$, $U_1$ and $U_2$. For a single realization, we plot in Figure~\ref{fig:longtimes.png} the evolution of the radial wave-action spectrum, averaged over frequency shells:
\[ n^{\mathrm{rad}}(t,\mathrm{k}) \coloneqq \frac{L}{2\pi} \sum_{k\in \Gamma_{\mathrm{k}}} \left[ |\widehat{U_1}(k)|^2 + 2 \Re \left( \overline{\widehat{U_0}(k)} \widehat{U_2}(k) \right) \right],   \]
where $\Gamma_{\mathrm{k}}$ denotes the circular shell of thickness $2\pi / L$ centered at frequency $\mathrm{k}$. The results suggest the emergence of a direct energy cascade, characterized by a power-law decay close to $\mathrm{k}^{-2}$ at high wavenumbers. Furthermore, we compute the $L^{\infty}_t L^2_x$ norm of the first-order term $\Re ( \overline{\widehat{U_0}} \widehat{U_1} )$ which is equal to $2.8028 \times 10^{-4}$, illustrating the expected cancellation at first order in the Picard expansion.

\begin{figure}[H]
	\centering
	\captionsetup{width=0.75\textwidth}
		\includegraphics[width=1\textwidth,trim = 0cm 0cm 0cm 0cm, clip]{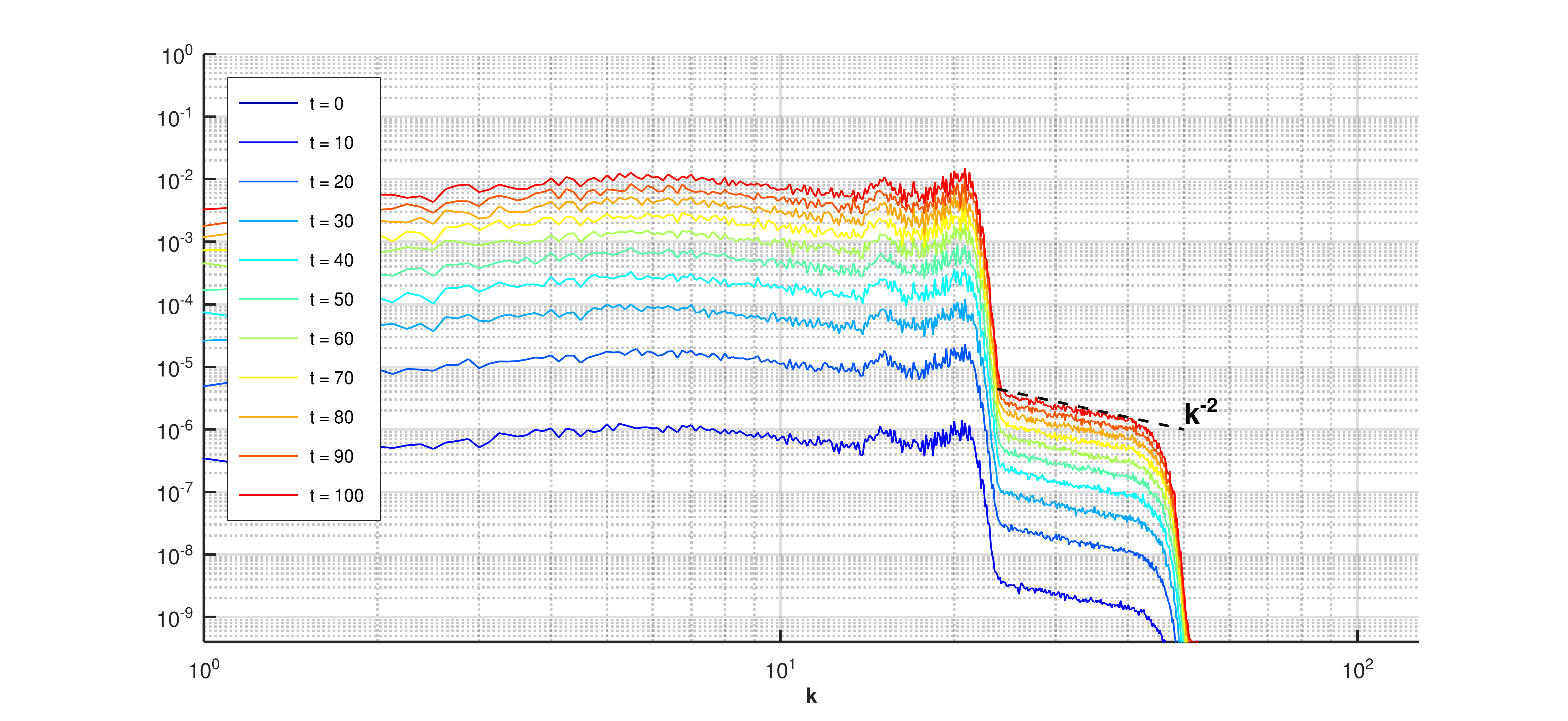}	
	\caption{Evolution of the wave spectra average $\mathrm{k} \mapsto n^{\mathrm{rad}}(t,\mathrm{k})$ for several times.}
	\label{fig:longtimes.png}
\end{figure}

\vspace{0.5cm}

\noindent \textcolor{gray}{$\bullet$} Q. Chauleur -- Univ. Lille, CNRS, Inria, UMR 8524 - Laboratoire Paul Painlevé, F-59000 Lille, France.\\
{\it E-mail}: quentin.chauleur@inria.fr

\noindent \textcolor{gray}{$\bullet$} A. Mouzard -- Modal’X - UMR CNRS 9023, Université Paris Nanterre, 92000 Nanterre, France.\\
{\it E-mail}: antoine.mouzard@math.cnrs.fr

\end{document}